\documentclass[12pt,reqno]{amsart}
\usepackage{latexsym,amsmath}
\usepackage[left=3cm,top=2.5cm,right=3cm,bottom=2.5cm]{geometry}
\usepackage[usenames,dvipsnames]{color}
\usepackage{amsthm}
\usepackage{amssymb}
\usepackage{epsfig}
\usepackage{mathrsfs}
\usepackage{verbatim}
\usepackage[normalem]{ulem}
\DeclareMathOperator{\ex}{ex}

\def\P{\mathcal{P}}

\newcommand{\hm}[1]{\leavevmode{\marginpar{\tiny%
$\hbox to 0mm{\hspace*{-0.5mm}$\leftarrow$\hss}%
\vcenter{\vrule depth 0.1mm height 0.1mm width \the\marginparwidth}%
\hbox to 0mm{\hss$\rightarrow$\hspace*{-0.5mm}}$\\\relax\raggedright #1}}}

\allowdisplaybreaks[2]
\newtheorem{theorem}{Theorem}[section]

\newtheorem{question}{Question}
\newtheorem{lemma}[theorem]{Lemma}

\newtheorem{corollary}[theorem]{Corollary}
\newtheorem{proposition}[theorem]{Proposition}
\newtheorem{claim}[theorem]{Claim}

\theoremstyle{definition}

\newcommand\eps{\varepsilon}
\def\({\left(}
\def\){\right)}

\def\P{\mathcal{P}_{k,s}}
\definecolor{red}{rgb}{1, 0, 0}
\def\epsilon{\varepsilon}

\begin{document}
\overfullrule=5pt

\title{An extension of the rainbow Erd\H{o}s-Rothschild problem}

\author[C. Hoppen]{Carlos Hoppen}
\address{Instituto de Matem\'atica e Estat\'{i}stica, UFRGS -- Avenida Bento Gon\c{c}alves, 9500, 91501--970 Porto Alegre, RS, Brazil}
\email{choppen@ufrgs.br}

\author[H. Lefmann]{Hanno Lefmann}
\address{Fakult\"at f\"ur Informatik, Technische Universit\"at Chemnitz,
Stra\ss{}e der Nationen 62, 09111 Chemnitz, Germany}
\email{Lefmann@Informatik.TU-Chemnitz.de}

\author[D. Nolibos]{Denilson Nolibos}
\address{Instituto de Matem\'atica e Estat\'{i}stica, UFRGS -- Avenida Bento Gon\c{c}alves, 9500, 91501--970 Porto Alegre, RS, Brazil}
\email{denilsonnolibos@gmail.com}

\thanks{This work was partially supported by CAPES and DAAD via Probral (CAPES Proc.~88881.143993/2017-01 and DAAD~57391132). The first author acknowledges the support of CNPq~308054/2018-0), Conselho Nacional de Desenvolvimento Cient\'{i}fico e Tecnol\'{o}gico.}

\begin{abstract}
Given integers $r \geq 2$, $k \geq 3$ and $2 \leq s \leq \binom{k}{2}$, and a graph $G$, we consider $r$-edge-colorings of $G$ with no copy of a complete graph $K_k$ on $k$ vertices where $s$ or more colors appear, which are called $\P$-free $r$-colorings. We show that, for large $n$ and $r \geq r_0(k,s)$, the $(k-1)$-partite Tur\'an graph $T_{k-1}(n)$ on $n$ vertices yields the largest number of $\P$-free $r$-colorings among all $n$-vertex graphs, and that it is the unique graph with this property.
\end{abstract}

\maketitle

\section{Introduction}

Given a fixed graph $F$, the well-known \emph{Tur\'{a}n problem} for $F$ is concerned with the maximum number $\ex(n,F)$ of edges over all $F$-free $n$-vertex graphs, namely over all $n$-vertex graphs that do not contain $F$ as a subgraph. The graphs that achieve this maximum are called \emph{$F$-extremal}. When $F=K_{k}$ is the complete graph on $k \geq 3$ vertices, the unique $F$-extremal graph on $n$ vertices is the balanced, complete, $(k-1)$-partite graph $T_{k-1}(n)$, known as the \emph{Tur\'{a}n graph} for $K_k$~\cite{turan}. For general graphs (and hypergraphs) $F$, determining $\ex(n,F)$ and the corresponding extremal graphs is a very important problem and there is a vast literature related with it (more information may be found in F\"{u}redi and Simonovits~\cite{FS2014}, and in the references therein). 

An $r$-coloring of a graph $G$ is a function $f \colon E(G) \longrightarrow [r]$ that associates a color in $[r]=\{1,\ldots,r\}$ with each edge of $G$. Erd\H{o}s and Rothschild~\cite{Erd74} were interested in $n$-vertex graphs that admit the largest number of $r$-colorings such that \emph{every color class is $F$-free}. In particular, they conjectured that, for all $n \geq n_0(k)$, the number of $K_k$-free 2-colorings is maximized by the Tur\'{a}n graph $T_{k-1}(n)$. Note that $F$-extremal graphs are natural candidates for maximality, as their edges may be colored arbitrarily, which leads to $r^{\ex(n,F)}$ colorings. It is clear that the number of colorings might increase if we have more than $\ex(n,F)$ edges to color, but additional edges also produce copies of $F$, placing constraints on colorings of their edges. 

Regarding the Erd\H{o}s-Rothschild Conjecture, Yuster~\cite{yuster} gave an affirmative answer for $k=3$ and any $n \geq 6$, while Alon, Balogh, Keevash and Sudakov~\cite{ABKS} showed that, for $r \in \{2,3\}$ and $n \geq n_0$, where $n_0$ is a constant depending on $r$ and $k$, the Tur\'{a}n graph $T_{k-1}(n)$ is the unique optimal $n$-vertex graph for the number of $K_k$-free $r$-colorings. Recently, H\`{a}n and Jim\'{e}nez~\cite{HJ2018} obtained better bounds on $n_0$ using the Container Method. For $r \geq 4$, the answer is more complicated. Pikhurko and Yilma~\cite{PY12} found the graphs that admit the largest number of such colorings for $r=4$ and $k \in \{3,4\}$, which turn out to be balanced, complete, multipartite graphs that are not $K_k$-free. Botler et al.~\cite{Botleretal19} characterized the extremal graphs for $k=3$ and $r=6$, and they gave an approximate result for $k=3$ and $r=5$. Pikhurko, Staden and Yilma~\cite{PSY16} showed that at least one of the graphs with the largest number of colorings is complete multipartite.

Balogh~\cite{Bal06} was the first to consider $r$-colorings that avoid a copy of a graph $F$ colored in a non-monochromatic way. A similar problem was investigated by Hoppen and Lefmann~\cite{rbmx} and by~Benevides, Hoppen and Sampaio~\cite{BHS17}, who considered edge-colorings of a graph avoiding a copy of $F$ with a \emph{prescribed pattern}. Given a graph $F$, a \emph{pattern} $P$ of $F$ is a partition of its edge set. An edge-coloring of a graph $G$ is said to be \emph{$(F,P)$-free} if $G$ does not contain a copy of $F$ in which the partition of the edge set induced by the coloring is isomorphic to $P$. For instance, if the partition $P$ consists of a single class, $(F,P)$-free colorings avoid \emph{monochromatic} copies of $F$. On the other hand, if $P$ is the pattern where each edge of $F$ lies in a different class, $(F,P)$-free colorings avoid \emph{rainbow} copies of $F$. These colorings are known as \emph{Gallai colorings} when $F=K_3$.

Given the number of colors $r \geq 1$, a graph $F$ and a pattern $P$ of $F$, let $\mathcal{C}_{r,(F,P)}(G)$ be the set of all $(F,P)$-free $r$-colorings of a graph $G$. We write
$$c_{r,(F,P)}(n) = \max\left\{\, |\mathcal{C}_{r,(F,P)}(G)| \colon |V(G)| = n \, \right\},$$
and we say that an $n$-vertex graph $G$ is \emph{$(r,F,P)$-extremal} if $|\mathcal{C}_{r,(F,P)}(G)| = c_{r,(F,P)}(n)$. Most results about $c_{r,(F,P)}(n)$ involve monochromatic or rainbow patterns, more information may be found in~\cite{BL19,CGM20,linear,rainbow_triangle,rainbow_kn} and in the references therein. In particular, the work of~\cite{BHS17} implies that, for any such pattern, there is an extremal $(r,F,P)$-extremal graph that is complete multipartite. 

Here, we generalize this problem to colorings that avoid a family of patterns. Let $k \geq 3$, $r \geq 2$ and $s \leq \binom{k}{2}$ be positive integers. Given a graph $G$, we are interested in $r$-edge-colorings of $G$ with no copy of $K_k$ colored with $s$ or more colors, which are called $\P$-free $r$-colorings. Let $\mathcal{C}_{r,\P}(G)$ denote this set of $r$-colorings and let 
\begin{equation}\label{def1}
c_{r,\P}(n) = \max\left\{\, |\mathcal{C}_{r,\P}(G)| \colon |V(G)| = n \, \right\}.
\end{equation}

If $s=1$, finding the $n$-vertex graphs that achieve $c_{r,\mathcal{P}_{k,1}}(n)$ colorings is just a restatement of the Tur\'{a}n problem, as it is equivalent to finding a $K_k$-free $n$-vertex graph with the largest number of edges. If $s=\binom{k}{2}$, this is precisely the problem of finding $(r,K_k,P)$-extremal graphs, where $P$ is the rainbow pattern of $K_k$.

We note that previous results already give the full solution of this problem for $k=3$, at least for large $n$. Note that $1\leq s \leq 3$ in this case.  The case $s=1$ corresponds to the Tur\'an problem, so that $T_2(n)$ is the unique extremal configuration for all $n,r \geq 2$. In the case $s=2$, every triangle in a graph $G$ has to be monochromatic in a $\mathcal{P}_{3,2}$-free $r$-coloring of $G$. So, if any edge contained in a triangle is removed from $G$, the number of colorings does not decrease, which immediately implies that $c_{r,\mathcal{P}_{3,2}}(n)=\, |\mathcal{C}_{r,\mathcal{P}_{3,2}}(T_2(n))|$ for any $r \geq 2$ and $n \geq  2$. It is easy to show that $T_2(n)$ is the only $n$-vertex graph with this property\footnote{Analogously $c_{r,\mathcal{P}_{k,2}}(n)=\, |\mathcal{C}_{r,\mathcal{P}_{k,2}}(T_{k-1}(n))|$ for any $r \geq 2$ and $n \geq  2$.}. For $s=3$, Balogh and Li~\cite{BL19} proved that, for $n$ sufficiently large, the complete graph is the unique extremal configuration for $r \leq 3$ and $T_2(n)$ is the unique extremal configuration for $r \geq 4$. Bastos, Benevides and Han~\cite{BBH20} obtained related results and Hoppen, Lefmann and Odermann~\cite{rainbow_triangle} had previously established the extremality of $T_2(n)$ for $r \geq 5$).

The following states two easy facts about determining~\eqref{def1} and the $n$-vertex graphs that achieve extremality.
\begin{lemma}\label{lemma_simple}
Let $n \geq k \geq 3$, $s \leq \binom{k}{2}$ and $r \geq 2$ be integers. 
\begin{itemize}
\item[(a)] If $r<s$, then $c_{r,\P}(n)= |\mathcal{C}_{r,\P}(K_n)|=r^{\binom{n}{2}}$.

\item[(b)] If $c_{r,\P}(n)= |\mathcal{C}_{r,\P}(T_{k-1}(n))|$ and $1 \leq s'<s$, then 
$$c_{r,\mathcal{P}_{k,s'}}(n)= |\mathcal{C}_{r,\mathcal{P}_{k,s'}}(T_{k-1}(n))|.$$
\end{itemize}
\end{lemma}

\begin{proof}
Part (a) is trivial, as no $r$-coloring can produce a copy of $K_k$ colored with $s$ or more colors if $r<s$. In part (b), the hypothesis tells us that for any $n$-vertex graph $G$
$$c_{r,\P}(n)=|\mathcal{C}_{r,\P}(T_{k-1}(n))|\geq |\mathcal{C}_{r,\P}(G)|.$$ 
The conclusion now follows from the fact that $\mathcal{C}_{r,\mathcal{P}_{k,s'}}(G) \subseteq \mathcal{C}_{r,\P}(G)$ for any graph $G$ and that $\mathcal{C}_{r,\P}(T_{k-1}(n))=\mathcal{C}_{r,\mathcal{P}_{k,s'}}(T_{k-1}(n))$.
\end{proof}

The work~\cite[Theorem 1.2]{rainbow_kn} shows that, given $k \geq 4$ and the rainbow pattern $P$ of $K_k$, there is $r_0$ such that $c_{r,(K_k,P)}(n)=|\mathcal{C}_{r,(K_k,P)}(T_{k-1}(n))| $ for all $r \geq r_0$ and $n \geq n_0(r,k)$. With Lemma~\ref{lemma_simple}~(b), we deduce that, for any $k \geq 4$ and $s \leq \binom{k}{2}$, there is $r_0$ such that $c_{r,\P}(n)= |\mathcal{C}_{r,\P}(T_{k-1}(n))|$ for all $r \geq r_0$ and $n \geq n_0(r,k,s)$. However, this value of $r_0$  is superexponential in $k$, and the authors of~\cite{rainbow_kn} believed that this result should hold for much smaller values of $r_0$. 

By addressing a more general problem, we are able to obtain much better bounds for smaller values of $s$; moreover, our results lead to better bounds on $r_0$ in the case $s= \binom{k}{2}$, i.e., when only rainbow copies of $K_k$ are avoided. The main result in this paper is the following. For simplicity, we write it in terms of functions that will be defined in the next subsection.
\begin{theorem}\label{main_thm}
Let $k \geq 4$ and $ 2 \leq s \leq \binom{k}{2}$ be integers. Fix $r \geq r_{0}(k,s)$, defined in~\eqref{def_r0},~\eqref{def_r1} and~\eqref{def_r2} below for $s \leq s_0(k)$, $s_0(k)<s\leq s_1(k)$ and $s>s_1(k)$, respectively. There is $n_{0}=n_0(r,k,s)$ for which the following holds. Every graph $G=(V,E)$ on $n > n_{0}$ vertices satisfies 
$$|\mathcal{C}_{r,\P}(G)| \leq r^{\ex(n,K_{k})}.$$ 
Moreover, equality holds if and only if $G$ is isomorphic to $T_{k-1}( n)$.
\end{theorem}

To prove Theorem~\ref{main_thm}, we use a stability method that relies on the following result.
\begin{theorem}\label{stability_thm}
Let $k \geq 4$ and $ 2 \leq s \leq \binom{k}{2}$ be integers. Fix $r \geq r_{0}(k,s)$, defined in~\eqref{def_r0},~\eqref{def_r1} and~\eqref{def_r2} below for $s \leq s_0(k)$, $s_0(k)<s\leq s_1(k)$ and $s>s_1(k)$, respectively. For any $\delta > 0$, there is $n_{0}=n_0(\delta,r,k,s)$ as follows. If $G=(V,E)$ is a graph on $n > n_{0}$ vertices such that $|\mathcal{C}_{r,\P}(G)| \geq r^{\ex(n,K_{k})}$, then there is a partition $V=W_{1}  \cup \cdots \cup W_{k-1}$ such that at most $\delta n^{2}$ edges have both endpoints in a same class $W_i$.
\end{theorem}
It turns out that the value of $r_0(k,s)$ in the above statements is needed in our proof of Theorem~\ref{stability_thm}, and we do not believe that it is best possible. However, our results give much better dependency on $s$ and $k$ than the bound in~\cite{rainbow_kn}. Indeed, if $s \leq s_0(k)$, the value given for $r_0(k,s)$ is less than $(s-1)^2$. Moreover, if $s \leq s_1(k)$, the quantity $r_0(k,s)$ is less than $(s-1)^7$.

\subsection{The functions \mathversion{bold}$r_0(k,s)$, $s_0(k)$ and $s_1(k)$\mathversion{normal} }

In order to specify the quantities given in the statement of Theorems~\ref{main_thm} and~\ref{stability_thm}, we shall define some additional functions. For $j \in \{2, \ldots, k-1\}$, let
\begin{equation}\label{def_A}
A(k,j)= \binom{k}{2}-\ex(k, K_{j+1}) = \binom{\lfloor k/j \rfloor}{2} \left( \left\lfloor \dfrac{k}{j} \right\rfloor j +j - k \right) + \binom{\lceil k/j \rceil}{2} \left( k- \left\lfloor \dfrac{k}{j} \right\rfloor j  \right),
\end{equation}
which, by Tur\'an's Theorem, is the minimum number of edges that must be deleted from a complete graph $K_k$ to make it $j$-partite.  Let 
\begin{eqnarray*}
s_0(k)&=&A(k,2)+2=\binom{k}{2}-\left\lfloor \frac{k}{2} \right\rfloor \cdot  \left\lceil \frac{k}{2} \right\rceil+2 \label{def_s}\\
s_1(k)&=&\binom{k}{2} - \left\lfloor \frac{k}{2} \right\rfloor +2. \label{def_s1}
\end{eqnarray*}

For $s \leq s_0(k)$, let $i^\ast$ be the least value of $i$ such that $A(k,k-i) \geq s-2$. As it turns out, we have $i^\ast \leq \min\{s-2,k-2\}$. Let $r_0(k,s)$ be the least integer greater than
\begin{equation} \label{def_r0}
(s-1)^{\frac{k-1}{k-2}} \prod^{i^\ast}_{i=2}\left( s- A(k,k-i+1)-1 \right)^{ \frac{1}{(k-i-1)(k-i)}}
\end{equation}

For any fixed $s>s_0(k)$, we consider additional parameters. Let  $j \in [k-1]$ and  $2 \leq p \leq k-1$ be integers satisfying the following condition:
\begin{eqnarray} \label{eq:opt2}
b(k,p,j) = \mbox{min } \left\{ j\binom{p}{2}, \left\lfloor \frac{k}{p} \right\rfloor \binom{p}{2} +\binom{k-\lfloor k/p \rfloor p}{2} \right\}
\leq \binom{k}{2} - s +  2, 
\end{eqnarray}
and define
\begin{eqnarray}\label{eq:0pt1}
L(k,s,p,j) =1+ \frac{2p(k-1)}{j(p-1)}.
\end{eqnarray}

If $s_0(k)<s \leq s_1(k)$, let $p^\ast$ be the largest $p \geq 2$ such that~\eqref{eq:opt2} holds for $j=k-1$. Our choice of $s$ ensures that there is such a $p$. We define $r_0(k,s)$ as the least integer greater than
\begin{equation} \label{def_r1}
 \left( s-A(k,2)-1\right)^{L(k,s,p^\ast,k-1)}  \cdot \left( \prod^{k-2}_{i=2}\left( s-A(k,k-i+1)-1  \right)^{ \frac{1}{(k-i-1)(k-i)}} \right) \cdot (s-1)^{\frac{k-1}{k-2} } .
\end{equation}
If $s > s_1$, let $j^\ast$ be the largest $j \geq 1$ such that~\eqref{eq:opt2} holds for $p=2$. We define $r_0(k,s)$ as the least integer greater than
\begin{equation} \label{def_r2}
\left( s-A(k,2)-1\right)^{L(k,s,2,j^\ast)}  \cdot \left( \prod^{k-2}_{i=2}\left( s-A(k,k-i+1)-1  \right)^{ \frac{1}{(k-i-1)(k-i)}} \right) \cdot (s-1)^{\frac{k-1}{k-2} }.
\end{equation}

Table~\ref{values} provides the values of $r_{0}(k,s)$ for a few values of $k$ and $s$\footnote{For completeness, we added values of $r_0$ known to hold for $k=3$.}. The symbols $\ast$ and $\star$ are used to indicate the first value of $s$ such that $s >s_0$ and such that $s>s_1$, respectively.
\begin{table}\label{values}
\begin{center}
	\begin{tabular}[h]{|c|c|c|c|c|c|c|c|c|c|c|c|c|} 
		\hline
		$ k \backslash s$	& 2  & 3  & 4  & 5  & 6 & 7 & 8 &  9 & 10 & 12& 13 & 15 \\ \hline
		$3$	& 2  & 4  &   &   & &  &  &  & &&& \\ \hline
		$4$	& 2  & 3  & 8  & 222$^\ast$  & 5434 &  &  &  & &&& \\ \hline
		$5$	& 2  & 3  & 5  & 11  & 19 & 457$^\ast$ & 3270 &  55507 &  218896 &&& \\ \hline
		$6$	& 2  & 3  &  5 & 7  & 15 & 24 & 35  & $606^\ast$  & 3528 &309393& 933907& $ 1.4 \cdot 10^{12 \star} $ \\ \hline
	\end{tabular}
\end{center}
\caption{$r_{0}(k,s)$ for some small values of $k$ and $s$.}
\end{table}

For comparison, it is easy to see that, if $r \leq r_{1}(k,s)=\lceil (s-1)^{(k-1)/(k-2)}-1\rceil$, then $|\mathcal{C}_{r,\P}(K_n)|>|\mathcal{C}_{r,\P}(T_{k-1}(n))|$ for large values of $n$, so that Theorem~\ref{main_thm} cannot possibly be extended to such values of $r$. In particular, these two tables imply that the values for $r_0(k,s)$ are best possible for $s=3$ and for $(k,s)=(5,4)$.
\begin{center}
\begin{table}
	\begin{tabular}[h]{|c|c|c|c|c|c|c|c|c|c|c|c|c|}
		\hline
		$ k \backslash s$	& 3  & 4  & 5  & 6 & 7 & 8 &  9 & 10 & 12& 13 & 15 \\ \hline
		$4$	&   2  & 5  & 7 & 11 &  &  &  & &&& \\ \hline
		$5$	&   2  & 4  & 6  & 8 & 10 & 13 & 15 & 18&&& \\ \hline
		$6$	&   2  & 3  & 5  & 7 & 9 & 11  & 13  & 15 &20&22& $ 27 $ \\ \hline
	\end{tabular}
	\caption{Value of $r_{1}(k,s)$ for some values of $k$ and $s$.}
\end{table}
\end{center}

The paper is structured as follows. In Section~\ref{sec_preliminaries}, we introduce the tools needed to prove our main results. We then prove Theorem~\ref{stability_thm} and~\ref{main_thm} in Sections~\ref{sec_stability} and~\ref{sec_main}, respectively.

\section{Preliminaries}\label{sec_preliminaries}

In this section, we fix the notation and introduce concepts and results used to prove our main results. We first state a well-known auxiliary lemma.
 \begin{lemma} \label{meulema}
If $\ell \geq 2$ and $G$ is a graph with $m$ edges, then $G$ contains an $\ell$-partite subgraph with more than $(\ell-1) m/\ell$ edges.
\end{lemma}

The next lemma generalizes a result of Alon and Yuster~\cite{AY}.
\begin{lemma} \label{lema_four} Fix $1 \leq j \leq k-1$ and $2 \leq p \leq k-1$. Let $0 < \gamma \leq \frac{j(p-1)}{2p(k-1)^2}$ and let $H''$ be a $(k-1)$-partite graph on $m$ vertices with partition $V(H'') = U_1 \cup \cdots \cup U_{k-1}$ and at least $ \ex(m, K_{k}) - \gamma m^2$
edges. If we add at least $\left(\frac{p(k-1)}{j(p-1)} + 1\right) \gamma m^2$ new edges to $H''$, then in
the resulting graph there is a copy of $K_{k}$ with at most $b(k,p,j)=\min \{j\binom{p}{2}, \lfloor k/p \rfloor \binom{p}{2} +\binom{k-\lfloor k/p \rfloor p}{2} \}$ new edges. Every such copy contains at most $p$ vertices in each class $U_i$ and each new edge in the copy connects two vertices of $K_{k}$ that lie in a same vertex class $U_i$ of $H''$.
\end{lemma}

\begin{proof} 
Let $H''$ be as in the statement of the lemma. Consider adding at least 
$$\left(\frac{p(k-1)}{j(p-1)} + 1\right) \gamma m^2$$ 
new edges to $H''$ to produce a graph $H'$. At least  $\frac{p(k-1)}{j(p-1)} \gamma m^2$ new edges have both endpoints in a same partition class. By an averaging argument, there exist $j$ classes $U_{i_1},\ldots,U_{i_j}$ containing at least $\frac{p}{p-1} \gamma m^2$ new edges. Indeed, if this were false, we would have
\begin{eqnarray*}
\binom{k-2}{j-1}\left(e_{H'}(U_1)+\cdots+e_{H'}(U_{k-1})\right) &=& \sum_{1 \leq i_1 <\cdots <i_j \leq k-1} \left(e_{H'}(U_{i_1})+\cdots+ e_{H'}(U_{i_j}) \right)\\
&<& \binom{k-1}{j} \cdot \frac{p}{p-1} \gamma m^2,
\end{eqnarray*}
which implies that $e_{H'}(U_1)+\cdots+e_{H'}(U_{k-1})<\frac{p(k-1)}{j(p-1)} \gamma m^2$, a contradiction. 

Let $\Gamma$ be the spanning subgraph of $H'$ with edges in $E_{H'}(U_{i_1}) \cup \cdots \cup E_{H'}(U_{i_j})$, so that $\Gamma$ contains at least $\frac{p}{p-1} \gamma m^2$ edges. By Lemma~\ref{meulema}, $\Gamma$ has a $p$-partite subgraph $\Gamma'$ with more than $\gamma m^2$ edges. We will refer to the edges of $\Gamma'$ as the \emph{new} edges. The sum of the number of edges in $\Gamma'$ with the number of edges in $H''$ is greater than $\ex(m, K_{k})$, hence there exists a copy of $K_{k}$ in the union of $H''$ and $\Gamma'$. Note that there cannot be $(p+1)$ or more vertices of this copy in a same class $U_i$, as this would produce a copy of $K_{p+1}$ where all edges are new, a contradiction. Therefore this copy of $K_k$ contains at most  $j \binom{p}{2}$ new edges.  On the other hand, at most $\lfloor k/p \rfloor$ classes may contain $p$ vertices, so that the number of new edges is at most $\lfloor k/p \rfloor \binom{p}{2} +\binom{k-\lfloor k/p \rfloor p}{2}$. As a consequence, the number of new edges is at most $\min \{j\binom{p}{2}, \lfloor k/p \rfloor \binom{p}{2} +\binom{k-\lfloor k/p \rfloor p}{2}\}$, as required.
\end{proof}

For later reference, we state Lemma~\ref{lema_four} in the special case $p=2$ and $j=1$, which is precisely the result of Alon and Yuster~\cite{AY}.
\begin{corollary}\label{lema_two}
Let $0 < \gamma \leq 1/(4(k-1))^2$ be fixed and let $H''$ be a $(k-1)$-partite graph on $m$ vertices with partition $V(H'') = U_1 \cup \cdots \cup U_{k-1}$ with at least $ \ex(m, K_{k}) - \gamma m^2$ edges. If we add at least $(2k-1) \gamma m^2$ new edges to $H''$, then in the resulting graph there is a copy of $K_{k}$ with exactly one new edge, which connects two vertices of $K_{k}$ in the same vertex class $U_i$ of $H''$.
\end{corollary}

\subsection{Regularity Lemma}

To prove our results we use an approach similar to the one from \cite{ABKS}, which is based on the Szemer\'{e}di Regularity Lemma~\cite{Sze78}. Let $G = (V,E)$ be a graph, and let $A$ and $B$ be two subsets of $V(G)$. If $A$ and $B$ are non-empty, define the density of edges between $A$ and $B$ by
$$d(A,B) = \frac{e(A,B)}{|A||B|},$$
where $e(A,B)$ is the number of edges with one vertex in $A$ and the other in $B$. (When $A=B$, we write $e(A,A)=e(A)$.)
For $\eps > 0$ the pair $(A,B)$ is called \emph{$\eps$-regular} if, for every subsets $X \subseteq A$ and $Y \subseteq B$ satisfying $|X| \geq \eps|A|$ and $|Y| \geq \eps|B|$, we have
$$|d(X,Y) - d(A,B)| < \eps.$$
An \emph{equitable partition} of a set $V$ is a partition of $V$ into pairwise disjoint classes $V_1,\ldots,V_m$ of almost equal size, i.e., $\arrowvert |V_i| - |V_j| \arrowvert \leq 1$ for all pairs $i,j$. An equitable partition of the set $V$ of vertices of $G$ into the classes $V_1,\ldots,V_m$ is called \emph{$\eps$-regular} if at most $\eps \binom{m}{2}$ of the pairs $(V_i, V_j)$ are not $\eps$-regular.

We shall use the following colored version of the Regularity Lemma~\cite{kosi} that will be particularly useful for our purposes.
\begin{lemma} \label{lrs vc}
For every $\eps > 0$ and every positive integer $r$, there exists a constant $M = M(\eps,r)$ such that the following property holds. If the edges of a graph $G$ of order $n > M$ are $r$-colored $E(G) = E_1 \cup \cdots \cup E_r$, then there is a partition of the vertex set $V(G) = V_1 \cup \cdots \cup V_m$, with $1/\eps \leq m \leq M$, which is $\eps$-regular simultaneously with respect to all graphs $G_i = (V,E_i)$ for all $i \in [r]$.
\end{lemma}

A partition as in Lemma \ref{lrs vc} will be called a \emph{multicolored $\eps$-regular partition}. Given such a partition $V_1 \cup \cdots \cup V_m$ and a constant $\eta>0$, we may also define the \emph{multicolored cluster graph} $H=H(\eta)$ associated with this partition and with a constant $\eta>0$: the vertex set is $[m]$ and $e = \{i,j\}$ is an edge of $H$ if the pair $(V_i,V_j)$ is $\eps$-regular with respect to all colors and the density between $V_i$ and $V_j$ is at least $\eta$ for at least one of the colors. Each edge $e=\{i,j\}$ in $H$ is assigned the list $L_e$ of colors $c$ such that $c$ appears with density at least $\eta$ between $V_i$ and $V_j$ in $G$. 

Given a colored graph $\widehat{F}$, we say that a multicolored cluster graph $H$ contains $\widehat{F}$ if $H$ contains a copy of the (uncolored) graph induced by $\widehat{F}$ for which the color of each edge of $\widehat{F}$ is contained in the list of the corresponding edge in $H$. More generally, if $F$ is a graph with color pattern $P$, we say that $H$ contains $(F,P)$ if it contains some colored copy of $F$ with pattern $P$.

Given colored graphs $\widehat{F}$ and $\widehat{H}$, a function $\psi\colon V(\widehat{F}) \rightarrow V(\widehat{H})$ is called a \emph{colored homomorphism} of $\widehat{F}$ in $\widehat{H}$ if, for every edge $e=\{i,j\} \in E(\widehat{F})$, the pair $\{\psi(i),\psi(j)\}$ is an edge of $\widehat{H}$ with the color of $e$. If $H$ is a multicolored cluster graph, it suffices that the color of $e$ lies in the list associated with the edge $\{\psi(i),\psi(j)\}$. In connection with these definitions, the following embedding result holds (for a proof, see~\cite{rainbow_kn}).
\begin{lemma} \label{homocol}
For every $\eta > 0$ and all positive integers $k$ and $r$, there exist $\varepsilon = \varepsilon (r,\eta, k) > 0$ and a positive integer $n_0(r,\eta, k)$ with the following property. Suppose that $\widehat{G}=(V,E)$ is an $r$-colored graph on $n > n_0$ vertices with a multicolored $\varepsilon$-regular partition $V = V_1 \cup \cdots \cup V_m$ which defines the multicolored cluster graph $H = H(\eta)$. Let $\widehat{F}$ be a $k$-vertex graph colored with $t \leq r$ colors. If there exists a colored homomorphism $\psi$ of $\widehat{F}$ into $H$, then the graph $\widehat{G}$ contains $\widehat{F}$.
\end{lemma}

The following standard embedding result will also be useful.
\begin{lemma} \label{abund}
Let $ k \geq 2$ be an integer and fix a constant $ 0 < \alpha \leq 3/4$. Let $ G $ be a graph whose vertex set contains mutually disjoint sets $W_{1},\ldots , W_{k}$ with the following property. For every pair $ \{i,j\} \subseteq [k]$, where $i\neq j$, and all subsets $X_{i} \subseteq W_{i}$, where $| X_{i} |\geq \alpha^{k} | W_{i} |$, and $X_{j} \subseteq W_{j}$, where $| X_{j} |\geq \alpha^{k} | W_{j} |$, there are at least $\alpha| X_{i} || X_{j} |$ edges between $X_{i}$ and $X_{j}$ in $G$. Then $G$ contains a copy of $K_{k}$ with one vertex in each set $W_{i}$.
\end{lemma}

\begin{proof} Our proof is by induction on $k$. For $k =2$ the result is trivial. Assume by induction that the statement holds for $k-1$, where $k \geq 3$. Now, fix 
\begin{equation}\label{eq_aux3}
\alpha \leq \frac{3}{4} < \min\left\{(\ell-1)^{-\frac{1}{\ell}} \colon 2 \leq \ell \leq k\right\}
\end{equation}
and let $ G $ be a graph whose vertex set contains mutually disjoint sets $W_{1},\ldots , W_{k}$ as in the statement of the lemma.

For all $i \in [k-1]$, let $W^{i}_{k} \subseteq W_{k}$ contain all vertices in $W_k$ with fewer than $\alpha| W_{i} |$ neighbors in $W_{i}$. Then we have $e(W^{i}_{k}, W_{i}) < \alpha|W^{i}_{k}| | W_{i} |$, so that $| W^{i}_{k} | < \alpha^{k} |W_{k}|$ by hypothesis. Our choice of $\alpha$ implies that  
$$\left| \bigcup^{k-1}_{i = 1} W^{i}_{k} \right| <(k-1)\cdot \alpha^{k} \cdot | W_{k} | \stackrel{\eqref{eq_aux3}}{<} | W_{k} |.$$ 
Let $v$ be a vertex in $W_k \setminus  \bigcup^{k-1}_{i = 1} W^{i}_{k}$ and, for $i \in [k-1]$, let $W'_{i}$ be the set of neighbors of $v$ in $W_{i}$, so that $| W'_{i} | \geq \alpha | W_{i} |$. Observe that, for all subsets $X_{i} \subseteq W'_{i}$ and $X_{j} \subseteq W'_{j}$ such that $| X_{i} |\geq \alpha^{k-1} | W'_{i} |\geq \alpha^{k} | W_{i} |$ and $| X_{j} |\geq \alpha^{k-1} | W'_{j} | \geq \alpha^{k} | W_{j} |$, there are at least $\alpha| X_{i} || X_{j} |$ edges between $X_i$ and $X_j$ in $G$. By induction, $G$ contains a copy of $K_{k-1}$ with one vertex in each set $W'_{i}$, where $1\leq i \leq k-1$. Adding $v$ creates a copy of $K_k$ in $G$ with one vertex in $W_{i}$ for each $i \in [k]$.
\end{proof}

\subsection{Stability}
Another concept that will be particularly useful in our paper are stability results in the sense of Erd\H{o}s and Simonovits~\cite{simonovits}. It will be convenient to use the following theorem by  F\"uredi~\cite{fu15}.
 \begin{theorem} \label{theorem:stability_furedi2} 
	 Let $G = (V,E)$ be a $K_{k}$-free graph on $m$ vertices. If $|E| = \ex (m, K_{k}) -t$ for some $t \geq 0$, then there exists a partition  $V= V_1 \cup \cdots \cup  V_{k-1}$
 with $\sum_{i = 1}^{k-1} e(V_i) \leq  t$. 
 \end{theorem}
 We recall the following bounds on the number of edges in the Tur\'an graph $T_{k-1}(m)$:
\begin{eqnarray}\label {eq:turan_number_1}
 && \frac{(k-2)m^2}{2(k-1)} - k +1 < \ex(m,K_{k}) \leq \frac{(k-2)m^2}{2(k-1)}.
\end{eqnarray}

For later use, we state the following fact about the size of the classes in a $(k-1)$-partite graph with a large number of edges (the easy proof is in~\cite{rainbow_kn}). 
 \begin{proposition} \label{prop:prop1}
Let $G=(V,E)$ be a $(k-1)$-partite graph on $m$ vertices with $(k-1)$-partition $V= V_1 \cup \cdots 
\cup V_{k-1}$. If, for some $t \geq (k-1)^2$, the graph $G$ contains at least $\ex (m, K_{k}) - t$
edges, then for each $i \in [k-1]$ we have
$$
\left||V_i| - \frac{m}{k-1} \right| 
%\leq \sqrt{\frac{2(k-1)}{k} \cdot t + 2(k-1)} 
< \sqrt{2t}.
$$
\end{proposition}

We also consider the \emph{entropy function} $H \colon [0,1] \rightarrow [0,1]$ given by $H(x) = -x \log_2 x - (1-x) \log_2(1-x)$  with $H(0) = H(1) = 0$. It is used in the well-known inequality
\begin{eqnarray} \label{eq:entropy1}
\binom{n}{\alpha n} \leq 2^{H(\alpha) n}
\end{eqnarray}
for all $0 \leq \alpha \leq 1$. It turns out that, for $x \leq 1/8$, we have:
\begin{align} \label{eq:entropy2}
H(x) \leq -2x \log_2 x.
\end{align}

\section{Proof of Theorem~\ref{stability_thm}}\label{sec_stability}

In this section, we prove Theorem~\ref{stability_thm}. For convenience, we restate it here.
\addtocounter{section}{-2}
\addtocounter{theorem}{2}
\begin{theorem}
Let $k \geq 4$ and $ 2 \leq s \leq \binom{k}{2}$ be integers. Fix $r \geq r_{0}(k,s)$, defined in~\eqref{def_r0},~\eqref{def_r1} and~\eqref{def_r2} below for $s \leq s_0(k)$, $s_0(k)<s\leq s_1(k)$ and $s>s_1(k)$, respectively. For any $\delta > 0$, there is $n_{0}=n_0(\delta,r,k,s)$ as follows. If $G=(V,E)$ is a graph on $n > n_{0}$ vertices such that $|\mathcal{C}_{r,\P}(G)| \geq r^{\ex(n,K_{k})}$, then there is a partition $V=W_{1}  \cup \cdots \cup W_{k-1}$ such that at most $\delta n^{2}$ edges have both endpoints in a same class $W_i$.
\end{theorem}
\addtocounter{section}{2}
\addtocounter{theorem}{-3}

\begin{proof}
Fix positive integers $ k \geq 4 $, $2 \leq s \leq \binom{k}{2}$ and $r \geq r_{0}(k,s)$ according to (\ref{def_r0}), (\ref{def_r1}) or (\ref{def_r2}), respectively, depending on the value of $s$. Let $ \delta >0 $. We shall fix a positive constant $\beta_0 < \frac{\delta}{8k+2}$ and consider $ \eta >0 $  sufficiently small to satisfy 
\begin{eqnarray}
r\eta &<& \frac{\delta}{2} \label{eq:eta1} 
\end{eqnarray}
and inequality~\eqref{eq_aux1}.

Let $n_{0}=n_{0}(r,\eta,k)$ and $ \varepsilon= \varepsilon(r,\eta,k)>0 $, where  $\varepsilon < \eta/2 $, be given by Lemma \ref{homocol} and let $ M=M(\varepsilon,r) $ be defined by Lemma~\ref{lrs vc}. Let $G$ be a graph on $n > \max \{ n_{0}, M \}$ vertices with at least $r^{\ex(n,K_k)}$ distinct $ (K_{k},\geq s) $-free $r$-colorings. Fix one such coloring.

By Lemma~\ref{lrs vc}, there exists a partition of $V(G)$ into $1/\eps \leq  m\leq M $ parts that is $\eps$-regular with respect to all $r$ colors. Let $ H=H(\eta)$ be the $m$-vertex multicolored cluster graph associated with this partition, where each edge $e$ has a non-empty list
$L_e$ of colors. We write $E_i=E_i(H)$ for the set of edges of $H$ for which $|L_e|=i$, and we let $e_i(H)=|E_i(H)|$, for $i \in [r]$.

We shall bound the number of $r$-colorings of $G$ that lead to the partition $ V(G)=V_{1} \cup \cdots \cup V_{m} $ and to the multicolored cluster graph $ H $. Given a color $ c \in[r] $, the number of $\eps$-irregular pairs $(V_i, V_j)$ with respect to the spanning subgraph $ G_{c} $ of $G$ with edge set given by all edges of color $c$ is at most $ \varepsilon \binom{m}{2} $. This leads to at most
\begin{equation}\label{104}
r \cdot \varepsilon \cdot \binom{m}{2} \cdot \left(\frac{n}{m}\right)^{2} \leq \frac{r \cdot \varepsilon}{2}\cdot  n^{2}
\end{equation}
edges between $\eps$-irregular pairs with respect to some color. By the definition of an $ \varepsilon $-regular partition and the fact that  $ m \geq  1 / \varepsilon$, there are at most
\begin{equation}\label{105}
m \cdot \left(\frac{n}{m}\right)^{2} = \frac{n^{2}}{m} \leq \varepsilon n^{2}
\end{equation}
edges with both endpoints in the same class $V_{i}$ for some $i \in [m]$. Finally, the number of edges with some color $ c $ connecting a pair $ (V_{i}, V_{j}) $ such that the density of the pair in $G_c $ is less than $ \eta $ is bounded above by
\begin{equation}\label{106}
r \cdot \eta \cdot \binom{m}{2} \cdot \left(\frac{n}{m}\right)^{2} < \frac{r \cdot \eta}{2} \cdot n^{2}.
\end{equation}
Combining equations $ (\ref{104}) - (\ref{106}) $ and using that $ \varepsilon < \eta /2 $, there are less than $r \eta n^{2}$ edges of any of these three types. They may be chosen in the $r$-coloring of $G$ in at most $\binom{n^{2}/2}{r \eta n^{2}}$ ways and may be colored in at most $r^{r \eta n^{2}}$ ways. 

As a consequence, the number of $r$-colorings of $ G $ that give rise to the partition $ V(G)=V_{1}  \cup \cdots \cup V_{m} $ and the multicolored cluster graph $ H(\eta) $ is bounded above by
\begin{eqnarray} {\label{01}}
\binom{n^{2}/2}{r\eta n^{2}} \cdot r^{r \eta n^{2}} \cdot \left( \prod_{e \in E(H)}|L_{e}|  \right)^{(\frac{n}{m})^{2}} & \stackrel{\eqref{eq:entropy1}}{\leq} &  2^{H(2r\eta)\frac{n^{2}}{2}}  \cdot r^{r \eta n^{2}} \cdot \left( \prod_{i=1}^{r}i^{e_{i}(H)}  \right)^{(\frac{n}{m})^{2}},
\end{eqnarray}
where $e_i(H)$ is the number of edges of $H$ whose lists have size equal to $i$.

Since $m \leq M$, there are at most $M^{n} $ distinct $\eps$-regular partitions $ V(G)=V_{1}  \cup \cdots \cup V_{m} $. Summing (\ref{01}) over all possible partitions and all possible multicolored cluster graphs $ H $, the number of  $ (K_{k}, \geq s) $-free $r$-colorings of $ G $ is at most
\begin{eqnarray} \label{02}
M^{n}\cdot 2^{H(2r\eta)\frac{n^{2}}{2}} \cdot  r^{r \eta n^{2}} \cdot \sum_{H}\left( \prod_{i=1}^{r}i^{e_{i}(H)}  \right)^{(\frac{n}{m})^{2}} . 
\end{eqnarray}
Note that, in this expression, we have $m=m(H)=|V(H)|$. We wish to bound the value of $ \prod_{i=1}^{r}i^{e_{i}(H)}$.

First observe that 
\begin{eqnarray} \label{eq:tu1}
  e_{s}(H)+ \cdots +e_{r}(H) &\leq&   \ex(m,K_k).
  \end{eqnarray}
Otherwise, by Tur\'an's Theorem, the multicolored cluster graph $H$ would contain a copy of $K_k$ such that every edge has a list of size at least $s$. This clearly induces a colored homomorphism of some pattern of $K_k$ with at least $s$ classes into $H$, so that by Lemma~\ref{homocol} $G$ would contain a copy of $K_k$ whose set of edges is colored with $s$ colors, a contradiction. More generally, we prove the following:
\begin{claim} \label{claim}
Let $H$ be a $ (K_{k},\geq s) $-free multicolored graph where $E_j=E_j(H)$ is the number of edges in $H$ with list of size $j$ and $e_j(H)=|E_j(H)|$ for any $j \in [r]$. For any fixed $i$ such that $A(k,k-i) \leq s-1$, there is no copy of $K_k$ for which all edges lie in $ E_{s-A(k,k-i)} \cup \cdots \cup E_{r}$ and at least $A(k,k-i)$ edges lie in $ E_{s}\cup \cdots \cup E_{r}$. Therefore, if $i \leq k-2$ and $A(k,k-i) \leq s-1$, the following inequality holds:
\begin{equation*}
\frac{k-i-1}{k-i}\cdot  \left(e_{s-A(k,k-i)}(H) + \cdots + e_{s-1}(H)\right) + e_{s}(H)+ \cdots +e_{r}(H) \leq \ex(m,K_k).
\end{equation*}
\end{claim}

Note that $A(k,k-i) = i$ for $i \leq \lfloor k/2 \rfloor$, so in this case the inequality in the claim becomes
\begin{equation*}
\frac{k-i-1}{k-i}\cdot  \left(e_{s-i}(H) + \cdots + e_{s-1}(H)\right) + e_{s}(H)+ \cdots +e_{r}(H) \leq \ex(m,K_k).
\end{equation*}

\begin{proof} 
Assume that there is a copy of $ K_{k} $ for which all edges lie in $ E_{s-A(k,k-i)} \cup \cdots \cup E_{r}$ and at least $A(k,k-i)$ edges lie in $ E_{s}\cup \cdots \cup E_{r}$. Let $p$ be the number of edges in $ E_{s-A(k,k-i)} \cup \cdots \cup E_{s-1}$ in this copy of $K_k$. Proceeding greedily (and starting with the edges in $E_{s-A(k,k-i)} \cup \cdots \cup E_{s-1}$), we may find a copy of $K_k$ such that $\alpha \geq \min\{s-A(k,k-i),p\}$ additional distinct colors appear in the edges $E_{s-A(k,k-i)} \cup \cdots \cup E_{s-1}$. If $\alpha \geq s-A(k,k-i)$, then at least $s-\alpha$ distinct colors may be chosen in the edges in $ E_{s}\cup \cdots \cup E_{r}$, as this union contains at least $A(k,k-i)$ edges, each with a list of size at least $s$. If $\alpha=p<s-A(k,k-i)$, then the number of edges of the copy in $ E_{s}\cup \cdots \cup E_{r}$ is $\binom{k}{2}-p \geq s-p$, so that at least $s-\alpha$ additional distinct colors may be chosen for edges in this set. In both cases, we get a copy of $K_k$ colored with $s$ or more colors, the desired contradiction.  

Next, assuming that $k-i \geq 2$ and that $A(k,k-i) \leq s-1$, let $E' \subseteq E_{s-A(k,k-i)} \cup \cdots \cup E_{s-1}$ be maximum with the property that the edges in $ E'$ induce a $(k-i)$-partite subgraph of $H$. The number of edges of $E'$ in a copy of $K_k$ is at most $\binom{k}{2}-A(k,k-i)$, so that $|E'|+e_{s}(H)+ \cdots + e_{r}(H) \leq \ex(m,K_k)$ by the previous discussion. By Lemma~\ref{meulema}, we know that $ |E'|\geq (k-i-1) \cdot\left|E_{s-A(k,k-i)} \cup \cdots \cup E_{s-1} \right|/(k-i)$, which gives the desired result. 
\end{proof}

For a multicolored cluster graph $H$, let 
\begin{equation}\label{def_beta}
 \beta= \beta(H)=  \frac{1}{m(H)^2}\left( \ex(m(H),K_k)  - \sum_{j=s}^{r}e_{j}(H)  \right) \geq 0.
\end{equation}
To find an upper bound in (\ref{02}) on the number of $\P$-free $r$-colorings of $G$ , we use~\eqref{eq:tu1} and ~\eqref{def_beta} in the product 
\begin{eqnarray}\label{eq_UB}
\left(\prod_{e \in E(H)}|L_{e}|\right)^{(\frac{n}{m})^{2}} &=& \prod_{i=2}^r i^{e_i(H)} \nonumber \\
&\leq& 2^{e_2(H)} \cdot 3^{e_3(H)} \cdots (s-1)^{e_{s-1}(H)} \cdot r^{\ex(n,K_k)-\beta(H) n^2}.
\end{eqnarray}
Maximizing this product is the same as maximizing 
\begin{equation}\label{eq_LP}
\log\left(2^{e_2(H)} \cdot 3^{e_3(H)} \cdots (s-1)^{e_{s-1}(H)} \right)=\log{2}\cdot e_2(H) +\cdots +\log(s-1)\cdot e_{s-1}(H).
\end{equation}
Moreover, with~\eqref{def_beta}, the inequalities in Claim~\ref{claim} lead to the following constraints. For all $i \in [k-1]$ such that $A(k,k-i) \leq s-1$, we get inequalities of the form
\begin{equation}\label{linear_constraints}
\frac{k-i-1}{k-i}\cdot ( e_{s-A(k,k-i)}(H)+ \cdots + e_{s-1}(H))  \leq \beta m^2.
\end{equation}

For $s \leq s_{0}(k) $, let $ i^{*} $ be the least value of $ i $ such that 
$$ s-A(k,k-i) \leq 2,$$ 
as defined in the introduction. The fact that $s \leq s_{0}(k)$ implies that $ i^{*} \leq k-2$. The constraints (\ref{linear_constraints}) for $i \in [i^{\ast}]$ may be written as
\begin{equation}\label{linear_constraints2}
\left\{
\begin{aligned}
\frac{k-2}{k-1}\cdot ( e_{s-A(k,k-1)}(H)+ \cdots + e_{s-1}(H)) & \leq &\beta m^2 \\
\frac{k-3}{k-2}\cdot ( e_{s-A(k,k-2)}(H)+ \cdots + e_{s-1}(H)) & \leq &\beta m^2 \\
\vdots \hspace{2cm}\vdots  \hspace{2cm} \vdots \hspace{2cm} & \vdots & \hspace{1cm} \vdots  \\
\frac{k-i^{*}-1}{k-i^{*}}\cdot ( e_{s-A(k,k-i^{*})}(H)+ \cdots + e_{s-1}(H)) & \leq &\beta m^2.
\end{aligned}
\right.
\end{equation}

This leads to a linear program with objective function~\eqref{eq_LP}, and constraints~\eqref{linear_constraints2} and $e_2(H),\ldots,e_{s-1}(H) \geq 0$ (and possibly $e_1(H)$ if $A(k,k-i^{*})\geq s-1$). It is easy to see that the optimum is obtained for $e_{s-1}(H)=\frac{k-1}{k-2} \cdot \beta m^2$ and $e_{s-A(k,k-i+1)-1}(H)= \left(\frac{k-i}{k-i-1}-\frac{k-i+1}{k-i} \right) \cdot  \beta m^2=\frac{1}{(k-i-1)(k-i)} \cdot \beta m^2$ for $ i \in \{2, \ldots, i^\ast \}$. Note that in this case we have
\begin{eqnarray} \label{eq:2beta}
\sum_{i=1}^{i^{*}} e_{s-A(k,k-i+1)-1}(H)= 2 \beta m^2.
\end{eqnarray}

To simplify the expressions below, for $k \geq 3$ and $2 \leq i \leq k-2$, define the quantities
\begin{equation}\label{eq_sxi}
s_i= s-A(k,k-i+1)-1 \textrm{ and } \xi_i=\frac{1}{(k-i-1)(k-i)}.
\end{equation}

Plugging the optimal solution of the linear program into~(\ref{02}), we obtain (for $s \leq s_0(k)$),
\begin{eqnarray} \label{03}
&&M^{n}\cdot 2^{H(2r\eta)\frac{n^{2}}{2}} \cdot  r^{r \eta n^{2}} \cdot \sum_{H}\left( \prod_{i=1}^{r}i^{e_{i}(H)}  \right)^{(\frac{n}{m})^{2}}  \leq \nonumber \\ 
&&r^{(H(2r\eta)+ 2r \eta) \frac{n^{2}}{2}} \cdot  \sum_{H} \left(\frac{ \left( \prod^{i^\ast}_{i=2} s_i^{ \xi_i} \right) (s-1)^{\frac{k-1}{k-2}} }{r}\right)^{\beta(H) n^{2}} r^{\ex(n,K_k)}. 
\end{eqnarray}

By our choice of $r_0=r_0(k,s)$ (see~\eqref{def_r0}), given any $\beta_0>0$, there is $\eta>0$ such that
\begin{equation}\label{eq_aux1}
r^{2H(2r\eta)+ 2r \eta} \left(\frac{ \left( \prod^{i^\ast}_{i=2} s_i^{ \xi_i} \right) \cdot (s-1)^{ \frac{k-1}{k-2} } }{r}\right)^{\beta_0} <1.
\end{equation}
Recall that we are using $\beta_0<\delta/(8k+2)$ and that $\eta>0$ satisfies~\eqref{eq_aux1} for this value of $\beta_0$.

We claim that there exists a multicolored cluster graph $H$ such that $\beta(H) < \beta_0$. Indeed, if this does not happen, the inequality~\eqref{03} would be bounded above by
\begin{eqnarray} \label{eq:107}
&&r^{(H(2r\eta)+ 2r \eta) \frac{n^{2}}{2}} \cdot 2^{rM^2/2} \left(\frac{ \left(  \prod^{i^\ast}_{i=2} s_i^{ \xi_i} \right) \cdot (s-1)^{ \frac{k-1}{k-2} } }{r}\right)^{\beta_0 n^{2}} \cdot r^{\ex(n,K_k)} \nonumber\\
&&\stackrel{(n \gg 1,\eqref{eq_aux1})}{<} r^{\ex(n,K_k)},
\end{eqnarray}
a contradiction (we are using that the number of distinct multicolored cluster graphs is bounded above by $2^{rM^2/2}$, which is less than $r^{H(2r\eta)n^2/2}$ for $n$ sufficiently large).

So, let $H$ be a multicolored cluster graph for which $\beta=\beta(H) < \beta_0$. Consider the spanning subgraph $H'$ of $H$ with edge set $E_{s} \cup \cdots \cup E_{r}$. This graph contains $\ex(m,K_k)-\beta m^{2}$ edges. We may apply Theorem~\ref{theorem:stability_furedi2} with $ t=\beta m^{2} $. By removing at most $\beta m^{2}$ edges of $ H' $, we produce a $ (k-1) $-partite subgraph $ H'' $. Let $U_{1}  \cup \cdots \cup U_{k-1} $ be the resulting partition of $V(H'')=V(H)$ such that  
\begin{equation}\label{eq_partition}
\sum_{i=1}^{k-1}e_{H'}(U_{i}) \leq \beta m^{2}.
\end{equation}
By Claim~\ref{claim} for $i=k-2$, we know that $ e_{2}(H) + \cdots +e_{s-1}(H) \leq 2 \beta m^{2}$ for $s \leq s_0$.  To bound $ e_{1}(H) $, we apply Corollary \ref{lema_two} to the $ (k-1) $-partite graph $ H''$. This lemma ensures that we may not have $ e_{1}(H) \geq (2k-1) \cdot 2\beta m^{2} =(4k-2) \beta m^{2}$, otherwise adding $E_1$ to $H''$ would produce a copy of $ K_{k} $ for which exactly one of the edges would have a list of size one and all other edges would have lists with $ s $ or more colors, leading to the forbidden pattern. Therefore we must have $ e_{1}(H) < (4k-2) \beta m^{2}$. This gives us an upper bound on the number of edges of $H$ with color lists of size up to $s-1$:
\begin{equation}\label{UB_edges}
e_{1}(H)+ \cdots + e_{s-1}(H) \leq  4k \beta m^{2}. 
\end{equation}	
	
Let $W_{i} = \bigcup_{j \in U_{i}}V_{j}$, where $ i \in [k-1] $. We shall prove that $ V(G) = W_{1} \cup \cdots \cup W_{k-1}$ satisfies the conclusion of the theorem. Edges of $G$ with both endpoints in a same class $W_i$ may come from three sources: edges of $G$ that are not represented in $H$; edges of $G$ in pairs $(V_s,V_t)$ such that $\{s,t\} \in E(H) \setminus E(H')$; edges of $G$ in pairs $(V_s,V_t)$ that correspond to edges in $E(H')$ with both endpoints in a same class $U_j$. 
 By~(\ref{104}) -- (\ref{106}) and \eqref{UB_edges}, we obtain
\begin{eqnarray*}
	\sum_{i=1}^{k-1} e_{G}(W_{i}) & \leq & r\eta n^2 + \left(\sum_{i=1}^{k-1} e_{H'}(U_{i})+e_{1}(H)+ \cdots + e_{s-1}(H) \right) \cdot \left( \frac{n}{m}\right)^{2}\\
& \leq & r\eta n^2 + \left(\beta_0 m^{2}+4k \beta_0 m^{2} \right) \cdot \left( \frac{n}{m}\right)^{2} \leq \delta n^2,
\end{eqnarray*}
by our choice of $\beta_0$ and $\eta>0 $.

We now consider the case when $s>s_0(k)$. All the inequalities in~\eqref{linear_constraints2} hold up to $i^\ast=k-2$, but in this case $s-A(k,2)>2$, so that the variables $e_2(H),\ldots,e_{s-A(k,2)-1}(H)$ are not bounded by the linear constraints. The constraints become
\begin{equation}\label{linear_constraints3}
\left\{
\begin{aligned}
\frac{k-2}{k-1}\cdot e_{s-1}(H) & \leq &\beta m^2 \\
\frac{k-3}{k-2}\cdot ( e_{s-A(k,k-2)}(H)+  e_{s-1}(H)) & \leq &\beta m^2 \\
\vdots \hspace{2cm}\vdots  \hspace{2cm} \vdots \hspace{2cm} & \vdots & \hspace{2cm} \vdots  \\
\frac{1}{2}\cdot ( e_{s-A(k,2)}(H)+ \cdots + e_{s-1}(H)) & \leq &\beta m^2.
\end{aligned}
\right.
\end{equation}

Consider $1\leq j \leq k-1$ and $2 \leq p \leq k-1$ Assume that $s \leq \binom{k}{2}-b(k,p,j) +2$, where $b(k,p,j)$ comes from Lemma~\ref{lema_four}. We divide the set of multicolored cluster graphs into two classes, according to whether $\beta(H) \geq j(p-1)/(4p(k-1)^2)$ or  $\beta(H) <    j(p-1)/(4p(k-1)^2)$.

If $\beta=\beta(H) \geq  j(p-1)/(4p(k-1)^2)$, we have, for $s_i$ and $\xi_i$ as in~\eqref{eq_sxi},
\begin{eqnarray}\label{04}
 && \prod_{i=1}^{r}i^{e_{i}(H)} \nonumber \\
&\stackrel{(\ref{eq:2beta})}{\leq}&  \left( s-A(k,2)-1 \right)^{\binom{m}{2}-(\ex(m,K_k) + \beta m^2)} \cdot \left(\frac{\left( \prod^{k-2}_{i=2}s_i^{\xi_i} \right) \cdot (s-1)^{ \frac{k-1}{k-2} } }{r}\right)^{\beta m^2} \cdot r^{\ex(m,K_k)} \nonumber\\
 &\leq&  \left(s-A(k,2)-1 \right)^{\frac{m^2}{2(k-1)}} \left(\frac{\left( \prod^{k-2}_{i=2}s_i^{\xi_i} \right) \cdot (s-1)^{ \frac{k-1}{k-2} } }{ \left(s-A(k,2)-1 \right) \cdot r}\right)^{\beta m^2} \cdot r^{\ex(m,K_k)} \nonumber\\
 &\leq& \left(\frac{\left( s-A(k,2)-1 \right)^{\frac{2p(k-1)}{j(p-1)}} \cdot \left(\prod^{k-2}_{i=2}s_i^{\xi_i} \right) \cdot (s-1)^{ \frac{k-1}{k-2} } }{ \left( s-A(k,2)-1\right) \cdot r}\right)^{\frac{j(p-1)m^2}{4p(k-1)^2}} \cdot r^{\ex(m,K_k)}. 
\end{eqnarray}

Next suppose that $\beta=\beta(H) < j(p-1)/(4p(k-1)^2)$. As above, let $H'$ be the spanning subgraph of $H$ with edge set $E_s \cup \cdots \cup E_r$ and let $H''$ be a maximum $(k-1)$-partite subgraph of $H'$. Note that $|E(H')|=\ex(m,K_k)-\beta m^2$. Apply Theorem~\ref{theorem:stability_furedi2} and define $\gamma>0$ such that
$$|E(H'')| =\ex(m,K_k)-\gamma m^2 \geq |E(H')|-\beta m^2 \geq \ex(m,K_k)- 2\beta m^2,$$
where $\beta \leq \gamma \leq 2 \beta$. In particular, the set $\widetilde{E}=E(H')-E(H'') \subset E_s \cup \cdots \cup E_r$ has cardinality $(\gamma-\beta)m^2$.
Let $E'=E_2 \cup  \cdots \cup E_{s - A(k,2) - 1}$. We may apply Lemma~\ref{lema_four} to the edge set $\widetilde{E} \cup E'$ for our values of $j$ and $p$ to this value of $ \gamma$. If $|E'|  + |\widetilde{E}| \geq \left(p(k-1)/(j(p-1))+1 \right) \gamma m^2$, then the multicolored cluster graph obtained by adding the edges in $E' \cup \widetilde{E}$ to $H''$ would contain a copy of $K_k$ with at most $b(k,p,j)$ edges in $E'$. If $|E' \cup \widetilde{E} |=1$, we would greedily build an $s$-colored copy of $K_k$, starting with the edge in $E' \cup \widetilde{E}$, a contradiction. If $|E' \cup \widetilde{E}| \geq 2$, we obtain a copy of $K_k$ with at least $(2+\min\{s-2,\binom{k}{2}-b(k,p,j)\})$ colors. This is a contradiction whenever $s \leq \binom{k}{2}-b(k,p,j) +2$, which is precisely the hypothesis in this case.

As a consequence, we have
\begin{equation*}
e_2(H)+\cdots+e_{s-A(k,2) - 1}(H)+|\widetilde{E}| \leq  \left(\frac{p(k-1)}{j(p-1)}+1 \right) \gamma m^2,
\end{equation*}
so that 
\begin{eqnarray}\label{new_equation}
e_2(H)+\cdots+e_{s-A(k,2) - 1}(H) &\leq&  \left(\frac{p(k-1)}{j(p-1)}+1 \right) \gamma m^2 - |\widetilde{E}| \nonumber \\
&=&  \frac{p(k-1)}{j(p-1)}\gamma m^2 + \beta m^2 \nonumber \\
&\leq& \left(\frac{2p(k-1)}{j(p-1)} +1 \right) \beta m^2.
\end{eqnarray}

Using an argument as in~\eqref{03}, but applying the additional inequality~\eqref{new_equation} to bound $e_2(H),\ldots,e_{s-A(2)-1}(H)$, we get, again with the notation in~\eqref{eq_sxi} and using the fact that $e_{s-A(k,2)}(H)+\cdots+e_{s-1}(H)=2 \beta m^2$, when the optimum solution of the linear program is achieved, it follows that
\begin{equation}\label{05}
\prod_{i=1}^{r}i^{e_{i}(H)} \leq \left(\frac{\left( s-A(k,2)-1\right)^{\frac{2p(k-1)}{j(p-1) }+1} \cdot \left( \prod^{k-2}_{i=2}s_i^{\xi_i} \right) \cdot (s-1)^{\frac{k-1}{k-2} } }{r}\right)^{\beta m^{2}} \cdot r^{\ex(m,K_k)}.
\end{equation}

It is clear that~\eqref{04} is less than~\eqref{05} for any fixed pair $(j,p)$. Therefore, by choosing $r_{0}(k,s)$ greater than
$$\left( s-A(k,2)-1\right)^{L_{opt}(k,s) }  \cdot \left( \prod^{k-2}_{i=2}\left( s-A(k,k-i+1)-1  \right)^{ \frac{1}{(k-i-1)(k-i)}} \right) \cdot (s-1)^{\frac{k-1}{k-2} },$$  
where $L_{opt}(k,s)$ is the least possible value of $L(k,s,p,j) = 1+ \frac{2p(k-1)}{j(p-1)}$ subject to conditions~(\ref{eq:opt2aa}) and~(\ref{eq:opt2b}), we may proceed as in the case $s \leq s_0(k)$. That is, we may again fix $\beta_0<\delta/(8k+2)$ and  choose $\eta>0$ appropriately, and consider whether $\beta(H) \geq \beta_0$ for all multicolored cluster graphs $H$ (in which case we reach a contradiction) or whether there is a multicolored cluster graph $H$ for which $\beta<\beta_0$ (in which case we get the desired partition), see the inequalities (\ref{03}), (\ref{eq_aux1}) and (\ref{eq:107}). 

To conclude the proof of Theorem~\ref{stability_thm}, we determine the value of $L_{opt}(k,s)$. This is the least value of
\begin{eqnarray}\label{eq:0pt11}
L(k,s,p,j) = 1+ \frac{2p(k-1)}{j(p-1)},
\end{eqnarray}
where $j \in [k-1]$, $p \in \{2, \ldots, k-1\}$ and
\begin{eqnarray} \label{eq:opt2aa}
 s &\in& \left[\binom{k}{2} - \lfloor k/2\rfloor \lceil k/2\rceil + 3,
\binom{k}{2}\right]
\end{eqnarray} 
 satisfy
\begin{eqnarray} \label{eq:opt2b}
b(k,p,j) = \mbox{min } \left\{ j\binom{p}{2}, \left\lfloor \frac{k}{p} \right\rfloor \binom{p}{2} +\binom{k-\lfloor k/p \rfloor p}{2} \right\}
\leq \binom{k}{2} - s +  2.
\end{eqnarray}
Clearly, choosing $j$ and $p$ large would be good to minimize~\eqref{eq:0pt11}. The following claim summarizes our result. 

\begin{claim}\label{claim:lopt}
For all $s>s_0$, the quantity $L_{opt}(k,s)$ satisfies the following.
\begin{itemize}
\item[(a)] If $s \leq s_1$, then there is $p^\ast \in \{2,\ldots,k-1\}$ such that $L_{opt}(k,s)=L(k,s,p^\ast,k-1)$. Moreover, 
$$3< L_{opt}(k,s) \leq 5$$
\item[(b)] If $s > s_1$, then there is $j^\ast \in \{2,\ldots,k-2\}$ such that $$L(k,s,2,j^\ast)=L_{opt}(k,s) = 1 +  \frac{4(k-1)}{\binom{k}{2} - s +2 } > 9.$$
\end{itemize}
\end{claim}

To prove the claim, we start with part (b), where $s \geq \binom{k}{2} - \left\lfloor \frac{k}{2} \right\rfloor +3$. First note that
$$
\left\lfloor \frac{k}{p} \right\rfloor \binom{p}{2} +\binom{k-\lfloor k/p \rfloor p}{2} \geq \left \lfloor \frac{k}{2} \right\rfloor>\binom{k}{2}-s+2.
$$
In particular, to satisfy~\eqref{eq:opt2b}, we need $b(k,p,j) = j \binom{p}{2} \leq \binom{k}{2}-s+2$. As a consequence, for any pair $(p,j)$ such that~\eqref{eq:opt2b} holds, we have
$$L(k,s,p,j) \geq 1+ \frac{2p(k-1)}{(p-1) \left(\binom{k}{2}-s+2 \right)/\binom{p}{2}}=1+\frac{p^2(k-1)}{\binom{k}{2}-s+2} \geq L(k,s,2,\binom{k}{2}-s+2).$$ 
Since the pair $(p,j)=(2,\binom{k}{2} - s +2)$ satisfies~\eqref{eq:opt2b}, we deduce that $L_{opt}(k,s)=L(k,s,2,\binom{k}{2}-s+2)$, so that
\begin{eqnarray*}
L_{opt}(k,s) &=&  1 +  \frac{4(k-1)}{\binom{k}{2} - s +2 } \\
&\geq&  1 + \frac{4(k-1)}{ \binom{k}{2} - (\binom{k}{2} - \left\lfloor \frac{k}{2} \right\rfloor +3) + 2 }\\
&\geq &1 + \frac{8(k-1)}{k-2}  >9.
\end{eqnarray*}
This proves part (b). 

We now consider part (a). Fix $p \geq 2$ and fix a pair $(p,j)$ that satisfies~\eqref{eq:opt2b}. Since
$$
\left\lfloor \frac{k}{p} \right\rfloor \binom{p}{2} +\binom{k-\lfloor k/p \rfloor p}{2}
\leq \left\lfloor \frac{k(p-1)}{2} \right\rfloor,
$$
if the minimum in~\eqref{eq:opt2b} is attained by $ j\binom{p}{2}$, then 
$$
 j\binom{p}{2} \leq  \frac{k(p-1)}{2},
 $$
which implies $j \leq k/p$. This implies that, if there exists $j>k/p$ such that $(p,j)$ satisfies~\eqref{eq:opt2b}, then we must have
$$ \left\lfloor \frac{k}{p} \right\rfloor \binom{p}{2} +\binom{k-\lfloor k/p \rfloor p}{2} \leq \binom{k}{2} - s +  2,$$
so that the pair $(p,k-1)$ satisfies~\eqref{eq:opt2b} and leads to $L(k,s,p,k-1) \leq L(k,s,p,j)$.

For $j=k-1$, let $p^*$ be the largest value of $p$ such that $(p,k-1)$ satisfies~\eqref{eq:opt2b}. This is well defined because $(2,k-1)$ satisfies~\eqref{eq:opt2b} for $s \leq  \binom{k}{2} - \left\lfloor \frac{k}{2} \right\rfloor +2=s_1(k)$. Towards finding a suitable pair $(p,j)$ that minimizes  $L(k,s,p,j)$, the only other candidates are pairs $(j',p')$, satisfying~\eqref{eq:opt2b}, such that $j' \leq k/2$ and $p^* \leq p' \leq k-1$. 

The inequality 
\begin{equation}\label{eq_aux4}
\frac{p \cdot \frac{k}{2}}{(p-1)(k-1)} \leq  \frac{p'}{p'-1}.
\end{equation}
holds because the left-hand side is at most $k/(k-1)$, while the right-hand side is greater than this, as $p' \leq k-1$. We conclude that
$$
L(k,s,p^*,k-1)=\frac{2p^* (k-1)}{(k-1) (p^* -1)} \stackrel{\eqref{eq_aux4}}{\leq} \frac{2p'(k-1)}{\frac{k}{2}\cdot (p'-1)} \leq \frac{2p'(k-1)}{j'\cdot (p'-1)}=L(k,s,p',j').
$$
Thus, to compute $L_{opt}(k,s)$, it remains to find the right value of $p^*$. 

For $j=k-1$ and $p\geq 2$, we have
$$
(k-1) \binom{p}{2} \geq \frac{k(p-1)}{2},
$$
hence in this case
$$
b(k,p,k-1) = \left\lfloor \frac{k}{p} \right\rfloor \binom{p}{2} +\binom{k-\lfloor k/p \rfloor p}{2} \leq \left\lfloor\frac{k(p-1)}{2}\right\rfloor.
$$
 This means that  (\ref{eq:opt2b}) becomes
\begin{equation}\label{eq_aux2}
\left\lfloor \frac{k}{p} \right\rfloor \binom{p}{2} +\binom{k-\lfloor k/p \rfloor p}{2}  \leq \binom{k}{2} - s +2
\end{equation}
and that 
$$L_{opt}(k,s)=L(k,s,p,k-1)=1+\frac{2p^*}{p^*-1}$$
where $p^*$ is the largest $p$ satisfying~\eqref{eq_aux2}.

On the one hand,
$$L(k,s,p^\ast,k-1) = 1+ \frac{2p^\ast}{p^\ast-1} > 3$$
for any $p^\ast \geq 2$.
On the other hand, we have
$$
L_{opt}(k,s) \leq L(k,s,2,k-1)=5.
$$
 This establishes part (a) of our claim, and finishes the proof of Theorem~\ref{stability_thm}.
\end{proof}

\section{Proof of Theorem~\ref{main_thm}}\label{sec_main}

To prove Theorem~\ref{main_thm}, we shall use the following special case of an auxiliary result~\cite{B?HLN}, whose proof uses tools as in~\cite[Theorem~1.1]{BHS17}.
\begin{theorem}\label{thm_atleast2}
Let  $n,r,k \geq 2$ with $2 \leq s \leq \binom{k}{2}$ be integers. If there exists an $(r, \P)$-extremal graph on $n$ vertices that is not complete multipartite, then there exist at least two non-isomorphic $(r,\P)$-extremal complete multipartite graphs on $n$ vertices. 
\end{theorem}

We now show that, for $k \geq 4$, $2 \leq s \leq \binom{k}{2}$, $r>r_0(k,s)$ and sufficiently large $n$, there is actually a single $(r, \P)$-extremal graph on $n$ vertices, the Tur\'{a}n graph $T_{k-1}(n)$. 
\begin{proof}[Proof of Theorem~\ref{main_thm}]
Let $k \geq 4$ and $ 2 \leq s \leq \binom{k}{2}$ be integers.   Fix $r \geq r_{0}(k,s)$, with $r_0$ defined in~\eqref{def_r0},~\eqref{def_r1} and~\eqref{def_r2} for $s \leq s_0(k)$, $s_0(k)<s\leq s_1(k)$ and $s>s_1(k)$, respectively.  Consider a constant $0<\alpha \leq 3/4$ such that
\begin{equation}\label{def_alpha}
2^{2 H(\alpha)} \cdot (r-1) < r,
\end{equation}
and fix $\delta>0$ such that
\begin{equation}\label{def_delta}
\delta<\frac{1}{2(k-1)^8} \textrm{ and }r^{\delta} < \left[\frac{r}{(r-1)\cdot 2^{2 H(\alpha)}}\right]^{\frac{\alpha^{2(k-1)}}{r^{2}(k-1)^{6}}}.
\end{equation}
Let $n_0=n_0(r,k,s)$ from Theorem~\ref{stability_thm}. We shall further assume that $n_0$ is large enough so that all the inequalities marked with $n \gg 0$ are satisfied. Fix $n_1 \geq n_0^2$.

To reach a contradiction, suppose that there is an $n$-vertex graph $G=(V,E)$ that is $(r,K_k,\geq s)$-extremal, but $G \neq T_{k-1}(n)$. We may assume that $G$ is a complete multipartite graph (if it is not, replace it by an $(r,K_k,\geq s)$-extremal graph that is complete multipartite and different from $T_{k-1}(n)$, which exists by Theorem~\ref{thm_atleast2}). Let $V=V_1' \cup \cdots \cup V_p'$ be the multipartition of $G$, where $p \geq k$.

Let $V=V_1 \cup \cdots \cup V_{k-1}$ be a partition of the vertex set of $G$ such that $\sum_{i=1}^{k-1} e(V_i)$ is minimized, so that
$$\sum_{i=1}^{k-1} e(V_i) \leq \delta n^2$$ 
by Theorem~\ref{stability_thm}. The minimality of this partition ensures that, if $v \in V_i$, then $|V_j \cap N(v)| \geq |V_i \cap N(v)|$, for every $j \in [k-1]$, where $N(v)$ denotes the set of neighbors of $v$. Moreover, by Proposition~\ref{prop:prop1}, we must have
$$\left| |V_{i}|-\dfrac{n}{k-1}\right| < \sqrt{2 \delta} \ n.$$
Finally, given that $p > k-1$, there must be an edge $\{x,y\} \in E$ whose endpoints are contained in the same class of the partition; assume without loss of generality that they lie in $V_{k-1}$ and that $|N(x) \cap V_{k-1}| \geq |N(y) \cap V_{k-1}|$. Since $x$ and $y$ are in different classes of $V_1' \cup \cdots \cup V_p'$, any $z \in V_{k-1} -\{x,y\}$ must be adjacent to $x$ or $y$. We conclude that, for any $i \in [k-1]$,
$$|N(x) \cap V_{i}| \geq |N(x) \cap V_{k-1}| \geq \frac{|V_{k-1}|-2}{2}+1 \geq \dfrac{n}{2(k-1)}-\frac{\sqrt{2 \delta}}{2} \ n \stackrel{\eqref{def_delta}}{\geq} \frac{n}{(k-1)^3}.$$
For simplicity, we write $W_i=N(x) \cap V_i$ for $i \in [k-1]$. 

We shall consider the cases $ 2 \leq s \leq \binom{k-1}{2}+1 $ and $s > \binom{k-1}{2}+1$ separately.

\vspace{5pt}

\noindent \textbf{Case 1.} Assume that $ 2 \leq s \leq \binom{k-1}{2}+1$.  Let $\mathcal{C}$ be the family of $\P$-free $r$-colorings of $G$. 

Fix a coloring $\widehat{G} \in \mathcal{C}$. For each $i \in [k-1]$ and each color $c \in [r]$, let $W^{\widehat{G}}_{i,c}$ be the set of vertices in $N(x) \cap V_i$ that are connected to $x$ by an edge of color $c$ in $\widehat{G}$. By the pigeonhole principle, for each $i \in [k-1]$, there must be a color $c_i \in [r]$ such that 
$$|W^{\widehat{G}}_{i,c_i}| \geq \frac{|W_i|}{r} \geq \frac{n}{r(k-1)^{3}}.$$

We say that color $c$ is rare with respect to a pair $\{i,j\} \in \binom{[k-1]}{2}$ if there exist subsets $X_{i} \subseteq W^{\widehat{G}}_{i,c_i}$ and $X_{j} \subseteq W^{\widehat{G}}_{j,c_j}$, where $|X_i| \geq \alpha^{k-1} |W^{\widehat{G}}_{i,c_i}|$ and $|X_j| \geq \alpha^{k-1} |W^{\widehat{G}}_{j,c_j}|$, for which the number of edges of color $c$ between $X_i$ and $X_j$ is less than $\alpha| X_{i} || X_{j} |$. Otherwise, $c$ is said to be abundant for the pair $\{i,j\}$. 

We claim that, for any fixed $\widehat{G}$, there must be a pair $\{i,j\}  \in \binom{[k-1]}{2}$ and a color $c \in [r]$ such that $c$ is rare with respect to $\{i,j\}$. To see why this is true, assume on the contrary that every color is abundant with respect to every pair. Since $r \geq s$ and $s \leq \binom{k-1}{2}+1$, we may choose $s-1$ colors in $[r] \setminus \{c_{k-1}\}$ and assign them arbitrarily to pairs $\{i,j\} \in \binom{k-1}{2}$ in a way that each pair is assigned a color and all colors appear. This leads to an edge-coloring $\widehat{K_{k-1}}$ of a copy of $K_{k-1}$ with vertex set $\{v_1,\ldots,v_{k-1}\}$ where exactly $s-1$ colors appear and they are all different from $c_{k-1}$. By Lemma~\ref{abund}, $\widehat{G}$ contains a copy of $\widehat{K_{k-1}}$ with vertex set $\{x_1,\ldots,x_{k-1}\}$ with the property that $x_i \in W^{\widehat{G}}_{i,c_i}$ for all $i \in [k-1]$. By construction, we see that $\widehat{G}[x,x_1,\ldots,x_{k-1}]$ induces a copy of $K_k$ where at least $s$ colors appear, the desired contradiction.

We are now ready to find an upper bound on $\mathcal{C}$. To do this, we shall bound the number of $\P$-free $r$-colorings that may be associated with a pair $(X_i,X_j)$ and color $c$ as above. There are $r$ choices for the color $c$ and at most $2^{2n}$ choices for the pair of sets $X_i,X_j$.
Once $c$ and the sets $X_i,X_j$ are fixed, we have at most 
\begin{equation*}
\binom{|X_{i}||X_{j}|}{\alpha|X_{i}||X_{j}|}(r-1)^{|X_{i}||X_{j}|}\stackrel{\eqref{eq:entropy1}}{<}2^{H(\alpha)|X_{i}||X_{j}|}(r-1)^{|X_{i}||X_{j}|}
\end{equation*}
ways to color the edges between $X_i$ and $X_j$. Note that $|X_i|,|X_j| \geq \frac{\alpha^{k-1} n}{r(k-1)^{3}}$ Assuming towards an upper bound that the at most $\ex (n, K_{k})+ \delta n^{2} - |X_{i}||X_{j}|$ remaining edges may be colored arbitrarily, we obtain
\begin{eqnarray}\label{eq_case1}
	|{\mathcal{C}}| &\leq & r\cdot 2^{2n} \cdot r^{\ex (n, K_{k})+ \delta n^{2} - |X_{i}||X_{j}|} \cdot 2^{H(\alpha)|X_{i}||X_{j}|}(r-1)^{|X_{i}||X_{j}|}\\
	&\leq&  \left( 2^{2 H(\alpha)} \cdot \dfrac{(r-1)}{r} \right)^{|X_{i}||X_{j}|} \cdot r^{\delta n^{2}} \cdot r^{\ex (n, K_{k})}  \nonumber  \\
	&\leq& \left( 2^{2H(\alpha)} \cdot \dfrac{(r-1)}{r} \right)^{\left(\frac{\alpha^{2k-2}}{r^{2}(k-1)^{6}}\right)n^{2}} \cdot r^{\delta n^{2}} \cdot r^{\ex(n, K_{k})}  \nonumber  \\
	& = & \gamma^{n^2} r^{\ex(n, K_{k})} < r^{\ex(n, K_{k})}  \nonumber 
\end{eqnarray}
where $\gamma<1$ is a constant\footnote{with respect to $n$.} by our choice of $\delta>0$ in~\eqref{def_delta}. This contradicts our choice of $G$.

\vspace{5pt}

\noindent \textbf{Case 2.} Now assume that $s> \binom{k-1}{2}+1$. Let $s'=s-\binom{k-1}{2}$, so that $2 \leq s'  \leq k-1$. To get a contradiction, we assume that the complete multipartite $(r,(K_k, \geq s))$-extremal graph $G$ has $r^{\ex(n,K_{k-1})+m}$ distinct $(K_k, \geq s)$-free colorings, where $m \geq 0$. We shall prove that the graph $G-x$ must have at least $r^{\ex (n-1, K_{k}) + m+1}$ such colorings. This conclusion will lead to the desired contradiction, as we could apply this argument iteratively until we obtain a graph $G'$  on $n_{0}$ vertices and  at least $ r^{\ex (n_{0}, K_{k}) + m+n-n_{0}}>r^{n^{2}_{0}}\geq r^{|E(G')|}.$

Let $\mathcal{C}$ be the family of $\P$-free $r$-colorings of $G$, and let $\mathcal{C}_1$ be the subfamily containing all colorings $\widehat{G}$ for which there is a choice of distinct indices $i_1, \ldots, i_{s'} \in [k-1]$ and distinct colors $ c_{i_{1}}, \ldots, c_{i_{s'}} \in [r] $ such that, for each $p \in [s']$, the set $ W^{\widehat{G}}_{i_p,c_{i_p}}\subseteq V_{p} \cap N(x)$ of neighbors of $x$ in $V_p$ through edges of color $c_{i_p}$ (with respect to $\widehat{G}$) satisfies $ |W^{\widehat{G}}_{i_{p},c_{i_p}}|\geq n / [r(k-1)^{3}] $. For any $i \in [k-1] \setminus \{i_1,\ldots,i_{s'}\}$, we fix an arbitrary color $c_i$ such that the set $W_{i,c_i}^{\widehat{G}}$ of neighbors of $x$ in $V_i$ through edges of color $c_{i}$ satisfies $ |W^{\widehat{G}}_{i,c_i}|\geq n / [r(k-1)^{3}]$ (this color exists by the pigeonhole principle). 

As in Case 1, we say that color $c$ is rare with respect to a pair $\{i,j\} \in \binom{[k-1]}{2}$ if there exist subsets $X_{i} \subseteq W^{\widehat{G}}_{i,c_i}$ and $X_{j,c_j} \subseteq W^{\widehat{G}}_{j}$, where $|X_i| \geq \alpha^{k-1} |W^{\widehat{G}}_{i,c_i}|$ and $|X_j| \geq \alpha^{k-1} |W^{\widehat{G}}_{j,c_j}|$, for which the number of edges of color $c$ between $X_i$ and $X_j$ is less than $\alpha| X_{i} || X_{j} |$. Otherwise, $c$ is said to be abundant for the pair $\{i,j\}$. We claim that, for any fixed $\widehat{G}$, there must be a pair $\{i,j\}  \in \binom{[k-1]}{2}$ and a color $c \in [r]$ such that $c$ is rare with respect to $\{i,j\}$. If this was not the case, we would be able to fix $s-s'=\binom{k-1}{2}$ distinct colors in $[r] \setminus \{c_{i_{1}}, \ldots, c_{i_{s'}}\}$ to be assigned to the edges of a copy of $K_{k-1}$, which, with Lemma~\ref{abund}, would lead to a contradiction. Using the arguments in~\eqref{eq_case1}, we conclude that $ |{\mathcal{C}_{1}}| \leq \gamma^{n^2} r^{\ex (n, K_{k})}$, where $\gamma<1$ is a constant.

As a consequence the family $\mathcal{C}_2=\mathcal{C}\setminus \mathcal{C}_1$ contains at least $r^{\ex (n, K_{k}) + m}-\gamma^{n^2} r^{\ex (n, K_{k})} \geq r^{\ex (n, K_{k}) + m-1}$ colorings. Fix a coloring $\widehat{G} \in \mathcal{C}_2$. We define a bipartite graph $B^{\widehat{G}}$ with bipartition $[k-1] \cup [r]$ such that $\{i,c\}$ is an edge if the set $ W^{\widehat{G}}_{i,c}\subseteq V_{i} \cap N(x)$ of neighbors of $x$ in $V_i$ through edges of color $c_{i}$ satisfies $ |W^{\widehat{G}}_{i,c}|\geq n / [r(k-1)^{3}]$. Note that $\widehat{G}$ lies in $\mathcal{C}_1$ if and only if $B^{\widehat{G}}$ contains a matching of size $s'$. Since $\widehat{G} \notin \mathcal{C}_1$, the following holds by Hall's Theorem. There is an integer $h$, where $1 \leq h \leq s'-1$, a set of distinct indices $I=\{i_1,\ldots,i_{h+1}\} \subseteq [k-1]$ and a set $C \subset [r]$ of colors, where $|C| \leq h$, such that $ |W^{\widehat{G}}_{i_p,c}|\geq n / [r(k-1)^{3}]$ only if $c \in C$. We shall associate each $\widehat{G} \in \mathcal{C}_2$ with such a triple $(h,I,C)$. Note that we may suppose that $|C|=h$ by adding arbitrary new elements to $C$ if necessary.

Let $\phi(\mathcal{C}_2)$ be the family that contains the projection of each element of $\mathcal{C}_2$ onto the edges incident with $x$. In other words, $\phi(\mathcal{C}_2)$ contains all $r$-edge colorings of the edges incident with $x$ that may be extended to a coloring in $\mathcal{C}_2$. We are now ready to find an upper bound on the cardinality of $\phi(\mathcal{C}_2)$. To do this, we shall bound the number of colorings that may be associated with a triple $(h,I,C)$ as above. Given $h$, there are $\binom{k-1}{h+1} \leq 2^{k-1}$ ways to fix $I=\{i_1,\ldots,i_{h+1}\}$ and $\binom{r}{h} \leq 2^r$ ways to choose $C$. Once $(h,I,C)$ is fixed, an upper bound on the number of ways to color the edges between $x$ and $V_{i}$, where $i \in I$, is
$$\binom{|V_i|}{n / [r(k-1)^{3}]}^r h^{|V_i|} \stackrel{\eqref{eq:entropy1}}{\leq} 2^{H\left(\frac{1}{r(k-1)^{3}}n\right)}  h^{|V_i|} \stackrel{\eqref{eq:entropy2}}{\leq} \left(r(k-1)^{3}\right)^{\frac{2n}{(k-1)^3}} h^{(\frac{1}{k-1}+\sqrt{2 \delta})n}.$$
The first term in the product is an upper bound on the number of ways to color edges using each of the at most $r$ colors that are used at most $n / [r(k-1)^{3}]$ times, the second term is an upper bound on the number of ways to color edges with the remaining colors, which lie in $C$. For $i \notin I$, the number of ways to color the edges between $x$ and $V_{i}$ will be bounded with the trivial bound $r^{|V_i|}\leq r^{(\frac{1}{k-1}+\sqrt{2 \delta})n}$.

Combining this information, we obtain
\begin{eqnarray} \label{Nx}
|\phi(\mathcal{C}_2)|
& \leq &  \sum_{h=1}^{s'-1} 2^{k-1} \cdot 2^r  \cdot [r(k-1)^{3}]^{\frac{2(h+1)n}{(k-1)^{3}}} \cdot (h^{h+1}\cdot r^{k-h-2})^{(\frac{1}{k-1}+\sqrt{2 \delta})n}
\end{eqnarray}
Consider the function $f(h)=h^{h+1}r^{-h}$. With $1 \leq h \leq k-2$, since $r \stackrel{(\ref{eq:4})}{>} (k-1)^4 > e(k+1)$ for $k \geq 4$,
 we have
$$\frac{f(h+1)}{f(h)}=\left(1+\frac{1}{h}\right)^h\cdot \frac{(h+1)^2}{rh} \leq \frac{e}{r} \left(h+2+\frac{1}{h} \right)<1.$$
This means that $f(h)$ is maximized for $h=1$. On the other hand, it is clear that $ [r(k-1)^{3}]^{\frac{2(h+1)n}{(k-1)^{3}}}\leq [r(k-1)^{3}]^{\frac{2s'n}{(k-1)^{3}}}$. Finally, the minimum degree $\delta_{k-1}(n)$ of a vertex in the Tur\'an graph $T_{k-1}(n)$ satisfies $\delta_{k-1}(n) \geq \frac{k-2}{k-1}n-1$. Using this in $ (\ref{Nx})$ leads to 
\begin{eqnarray}\label{Nx2}
|\phi(\mathcal{C}_2)| &\leq &s' \cdot 2^{k-1} \cdot 2^{r}\cdot [r(k-1)^{3}]^{\frac{2 s' n}{(k-1)^{3}}} \cdot r^{(k-3)(\frac{1}{k-1}+\sqrt{2 \delta})n}  \nonumber \\ 
&\stackrel{\eqref{def_delta}}{\leq}  & s' \cdot 2^{k-1} \cdot 2^{r} \cdot [r(k-1)^{3}]^{\frac{2 s' n}{(k-1)^{3}}} \cdot r^{(k-3)\left(\frac{1}{k-1}+\frac{1}{(k-1)^{4}}\right)n} \cdot r^{\delta_{k-1}(n)-\frac{k-2}{k-1}n+1}  \nonumber \\
& = & s' \cdot 2^{k-1} \cdot r \cdot 2^{r} \cdot \left[\frac{(k-1)^{6 s'}}{ r^{(k-1)^{2} - 2s'-\frac{k-3}{k-1} }}\right]^{\frac{n}{(k-1)^3}} \cdot r^{\delta_{k-1}(n)}.	
\end{eqnarray}
We claim that 
\begin{equation}\label{claimed_bound}
\frac{(k-1)^{6 s'}}{ r^{(k-1)^{2} - 2s'-\frac{k-3}{k-1} }}<1.
\end{equation}
Before proving this, we argue that it leads to the desired result. Using this, inequality~\eqref{Nx2} implies that
$$|\phi(\mathcal{C}_2)|  \leq r^{\delta_{k-1}(n)-2}$$
for sufficiently large $n$, and therefore
$$ |\mathcal{C}_{r,\P}(G-x)| \geq \frac{|\mathcal{C}_2|}{|\phi(\mathcal{C}_2)|} \geq \frac{r^{\ex (n, K_{k}) + m-1}}{r^{\delta_{k-1}(n)-2}}=r^{\ex (n-1, K_{k}) + m+1},$$
as required.

To conclude the proof, we show that~\eqref{claimed_bound} holds. We start with the case $k \geq 5$, where we first show that $r_0(k,s)>(k-1)^4$ when $s \geq \binom{k-1}{2}+2$. Indeed, by the definition of $r_0(s,t)$ (for $s >s_0$) and by the lower bounds on $L_{opt}(k,s)$ given in Claim~\ref{claim:lopt}, 
we deduce that
\begin{eqnarray}\label{eq_LL}
r_0(k,s) &>& (s-A(k,2)-1)^3  \cdot (s-A(k,2))^{2-\frac{k-1}{k-2}} \cdot (s-1)^{\frac{k-1}{k-2}} \nonumber \\
& \geq & (s-A(k,2)-1)^5 \cdot \left(\binom{k-1}{2}+1 \right) \nonumber \\
& > & \left(\binom{k-1}{2}-\left(\binom{k}{2}-\left\lfloor\frac{k}{2}\right\rfloor  \left\lceil\frac{k}{2} \right\rceil \right)+1\right)^5  \cdot \binom{k-1}{2} \nonumber \\
&\geq& \left(\left\lfloor\frac{k}{2}\right\rfloor  \left\lceil\frac{k}{2} \right\rceil -k+2 \right)^5 \cdot \binom{k-1}{2}.
\end{eqnarray}
It turns out that 
$$\left\lfloor\frac{k}{2}\right\rfloor \left\lceil\frac{k}{2} \right\rceil -k+2 \geq \frac{k^2-1}{4} -k+2 \geq k-2.$$
Therefore, the inequality~\eqref{eq_LL} leads to
\begin{eqnarray} \label{eq:4}
r_0(k,s) > \frac{(k-2)^{6}(k-1)}{2} > (k-1)^4,
\end{eqnarray}
where the last part may be easily proved by induction (for $k \geq 4$).

Coming back to~\eqref{claimed_bound}, we first consider the case $k\geq 5$. First note that 
$$\frac{(k-1)^{6 s'}}{ r^{(k-1)^{2} - 2s'-\frac{k-3}{k-1} }} \leq \frac{\left((k-1)^{4}\right)^{3s'/2}}{ r^{(k-1)^{2} - 2s'-1}}.$$
Since $r > (k-1)^{4}$, inequality~\eqref{claimed_bound} holds if we show that $3 s'/2 \leq (k-1)^{2} - 2s'-1$, which is equivalent to
$$\frac{7s'}{2} \leq k^2-2k.$$
The left-hand side of this inequality is at most $7(k-1)/2$, and it is indeed the case that $7(k-1) \leq 2k^2-4k$ for all $ k\geq 5$.

In the case $k=4$, inequality~\eqref{claimed_bound} holds if and only if
\begin{equation}\label{claimed_bound4}
3^{6s'} \leq r^{9-2s'-\frac{1}{3}}.
\end{equation}
Since we are in Case 2, we need to consider the cases $s=5$ (so $s'=2$) and $s=6$ (so $s'=3$). In the case $s=5$, the inequality holds if $3^{12} \leq r^{14/3}$, which holds for $r \geq 17$. According to Table~\ref{values}, $r_0(4,5)=222$. In the case $s=6$, inequality~\eqref{claimed_bound4} holds if $3^{18} \leq r^{8/3}$, holds for $r\geq 1662$. According to Table~\ref{values}, $r_0(4,6)=5434$. This concludes the proof of Theorem~\ref{main_thm}.
\end{proof}

\section{Final remarks and open problems}

In this paper, given integers $r \geq 2$, $n\geq k \geq 3$ and $2 \leq s \leq \binom{k}{2}$, we were interested in characterizing $n$-vertex graphs $G$ for which the number of $r$-edge-colorings with no copy of $K_k$ colored with $s$ or more colors satisfies
\begin{equation}\label{eq_remarks}
|\mathcal{C}_{r,\P}(G)|=c_{r,\P}(n) = \max\left\{\, |\mathcal{C}_{r,\P}(G')| \colon |V(G')| = n \, \right\}.
\end{equation}
This problem is a common generalization of the Tur\'{a}n problem and of the rainbow Erd\H{o}s-Rothschild problem, i.e.,  the problem of finding $n$-vertex graphs $G$ with the largest number of $r$-edge-colorings with no rainbow copy of $K_k$.

More precisely, we have found functions $r_1(k,s)$ and $r_0(k,s)$ such that 
\begin{itemize}
\item[(a)] If $r \geq r_0(k,s)$ and $n$ is sufficiently large, then $|\mathcal{C}_{r,\P}(G)|=c_{r,\P}(n)$ if and only if $G$ is isomorphic to $T_{k-1}(n)$.

\item[(b)] If $r \leq r_1(k,s)$, then $|\mathcal{C}_{r,\P}(K_n)| > |\mathcal{C}_{r,\P}(T_{k-1}(n))|$.
\end{itemize}
We should mention that reference~\cite{rainbow_kn} gives a function $r_0'(k)$ such that the statement of part (a) holds for all $r \geq r_0'(k) \geq \binom{k}{2}^{8k-4}$ when $s=\binom{k}{2}$. With Lemma~\ref{lemma_simple}, this implies the \emph{existence} of a function $ r_0'(k,s)$ such that (a) is satisfied. The results of the current paper are such that $r_0(k,\binom{k}{2})\leq  \left(k^2/4\right)^{4k} <r_0'(k)$, that $r_0(k,s) \leq (s-1)^2$ for $s\leq s_0(k)=\binom{k}{2}-\left\lfloor \frac{k}{2} \right\rfloor \cdot  \left\lceil \frac{k}{2} \right\rceil+2 $ and that $r_0(k,s) \leq (s-1)^7$ for $s\leq s_1(k)=\binom{k}{2} - \left\lfloor \frac{k}{2} \right\rfloor +2$.

For general values of $k$ and $s$, there is a significant gap between the functions $r_0$ and $r_1$ in (a) and (b), but it turns out that $r_0(k,s)=r_1(k,s)+1$ for all pairs $(k,3)$ with $k \geq 4$ and for infinitely many other pairs $(k,s)$. Computing the values of $r_0$ and $r_1$, we see that this happens for the following pairs if $(k,s)$ where $4 \leq k <10$ and $s \geq 4$:
$$(k,s) \in \{(5,4),(7,4),(7,5),(8,4),(8,5),(9,3),(9,4),(9,6)\}.$$ 
In fact, the following holds.

\begin{proposition}\label{prop1}
Let $s \geq 3$ be an integer. There exists $k_0$ such that, for all $k \geq k_0$, the pair $(k,s)$ satisfies $r_0(k,s)=s$ and $r_1(k,s)=s-1$.
\end{proposition}

\begin{proof}
Let $s \geq 3$. First notice that $(s-1)^{1/(k-2)}$ is not an integer for $k > 2 + \log_2(s-1)$. Fix $k_1$ such that $(s-1)^{1/(k-2)}$ is not an integer and $(s-1)^{\frac{k-1}{k-2}}<s-\frac{1}{2}$ for all integers $k \geq k_1$. In particular, we have 
$$s-1=\lceil (s-1)^{\frac{k-1}{k-2}}-1\rceil = r_1(k,s).$$ 

Next, let $k_2$ be such that, for all integers $k \geq k_2$, we have $s \leq s_0(k)$ and $i^\ast=s-2$, see~(\ref{def_r0}). In particular, $r_0(k,s)$ is the least integer greater than
$$(s-1)^{\frac{k-1}{k-2}} \cdot (s- 2)^{ \frac{1}{(k-3)(k-2)}} \cdot  (s- 3)^{ \frac{1}{(k-4)(k-3)} }\cdots 2^{\frac{1}{(k-s+1)(k-s+2)}} \leq (s-1)^{\frac{k-1}{k-2}} \cdot (s- 2)^{ \frac{s-3}{(k-s+1)(k-2)}}.$$
We may choose $k_3$ such that, for all $k \geq k_3$, it holds that
$$\left(1+\frac{1}{2(s-1)^{3/2}} \right)^{(k-s+1)(k-2)} > (s-2)^{s-3}.$$
This implies that, for all $k \geq \max\{k_1,k_2,k_3\}$, it is
$$(s-1)^{\frac{k-1}{k-2}}(s- 2)^{ \frac{s-3}{(k-s+1)(k-2)}} < (s-1)^{\frac{k-1}{k-2}}\left(1+\frac{1}{2(s-1)^{3/2}} \right) \leq (s-1)^{\frac{k-1}{k-2}}+\frac{1}{2}<s.$$

As a consequence, we have 
\begin{eqnarray*}s \leq r_0(k,s) &\leq& \left\lfloor (s-1)^{\frac{k-1}{k-2}}(s- 2)^{ \frac{s-3}{(k-s+1)(k-2)}} \right\rfloor+1< s+1,
\end{eqnarray*}
so that  $r_0(k,s)=s$.
\end{proof}

By definition, it is obvious that, if $r \leq s-1$, then $|\mathcal{C}_{r,\P}(G)|=c_{r,\P}(n)$ if and only if $G=K_n$. This implies that, for pairs $(k,s)$ such that $s=3$ or Proposition~\ref{prop1} is satisfied, the functions $r_1$ and $r_0$ are best possible in (a) and (b)\footnote{For sufficiently large $n$.}, and $K_n$ is $(r,\P)$-extremal for $r \leq 2$ and $T_{k-1}(n)$ is $(r,\P)$-extremal for $r \geq 3$.

We propose the following questions.
\begin{question}
We say that a pair $(k,s)$ such that $k\geq 3$ and $3 \leq s \leq \binom{k}{2}$ satisfies \emph{Property A} if there exists $r^*=r^*(k,s)$ such that, for any fixed $r \geq 2$, there exists $n_0$ for which the following holds: $G$ is an $(r,\P)$-extremal $n$-vertex graph with $n \geq n_0$ if and only if $r<r^*$ and $G=K_n$, or $r \geq r^*$ and $G=T_{k-1}(n)$. Is it true that all such pairs $(k,s)$ satisfy Property A?
\end{question}
The above paragraph ensures that all pairs that fulfil Proposition~\ref{prop1} also satisfy Property A. Moreover, previous work for triangles~\cite{BL19,lagos} implies that the pair $(3,s)$ satisfies Property A for any $s$, while the current paper shows that the pair $(k,3)$ satisfies Property A for any $k$. 

\begin{question}
Is it true that, for any pair $(k,s)$ such that $k\geq 3$ and $3 \leq s \leq \binom{k}{2}$ and any $r \geq s$, it holds that 
\begin{equation}\label{eq_kn}
|\mathcal{C}_{r,\P}(K_n)|=\left(\binom{r}{s-1}+o_n(1) \right) (s-1)^{\binom{n}{2}}?
\end{equation}
\end{question}
Equation~\eqref{eq_kn} holds for all $r \geq 3$ in the case $k=3$, see~\cite{BL19,BBH20} (the second reference actually shows that this relation holds for some functions $r=r(n)$ such that $r \rightarrow \infty$ as $n \rightarrow \infty$). We observe that, if a pair $(k,s)$ satisfies Property A and equation~\eqref{eq_kn} holds for $(k,s)$ and any $r > r_1(k,s)$, then we must have $r^*(k,s)= r_1(k,s)+1$.

In a different direction, given integers $r \geq 2$, $n\geq k \geq 3$ and $1 \leq s \leq \binom{k}{2}$, it would be interesting to consider $r$-edge-colorings with no copy of $K_k$ colored with $s$ or \emph{less} colors, which would lead to a common generalization of the Tur\'{a}n problem and of the (monochromatic) Erd\H{o}s-Rothschild problem. The class of extremal configurations for large $n$ and $r$ will be much richer, as  $K_n$  does not admit any such coloring by Ramsey's Theorem and $T_{k-1}(n)$ admits fewer colorings than other constructions (see~\cite{ABKS,PSY16}).

\end{document}